\documentclass[preprint,12pt,3p]{elsarticle}

\usepackage{amssymb}
\usepackage{amsmath,amsthm}
\usepackage{mathrsfs}
\usepackage{hyperref}
\usepackage{cleveref}
\usepackage[all,cmtip]{xy}
\usepackage{epsf, graphicx}
\usepackage{latexsym,amsfonts,amsbsy,amssymb}
\usepackage{geometry}
\usepackage{titletoc}
\usepackage{color,xcolor}
\usepackage{rotating}
\usepackage{algorithm}
\usepackage{algorithmic}
\usepackage{amscd}

\makeatletter

\@addtoreset{equation}{section} \makeatother
\newtheorem{theorem}{Theorem}[section]
\newtheorem{lemma}[theorem]{Lemma}
\newtheorem{notation}[theorem]{Notation}

\newtheorem{remark}[theorem]{Remark}
\newtheorem{void}[theorem]{}
\newtheorem{proposition}[theorem]{Proposition}

\def\Ind{{\rm Ind}}
\def\Res{{\rm Res}}

\def\br{{\rm br}}

\def\LP{\mathcal{LP}}

\def\Aut{{\rm Aut}}

\def\Hom{{\rm Hom}}

\def\End{{\rm End}}
\def\Gal{{\rm Gal}}
\def\stab{{\rm stab}}

\def\O{\mathcal{O}}

\def\P{\mathcal{P}}

\def\rk{{\rm rk}}
\def\Tr{{\rm Tr}}
\def\t{\tilde}

\makeatletter
\def\ps@pprintTitle{%
\let\@oddhead\@empty
\let\@evenhead\@empty
\def\@oddfoot{\reset@font\hfil\thepage\hfil}
\let\@evenfoot\@oddfoot
}
\makeatother

\begin{document}

\begin{frontmatter}

\title{On stable equivalences of Morita type with twisted diagonal vertices}

\author{Xin Huang}


\begin{abstract}
We give a new proof, by using simplified terminology and notation, to a result of Puig stating that if a bimodule of two block algebras of finite groups over an algebraically closed field induces a stable equivalence of Morita type and has a twisted diagonal vertex, then it has an endopermutation module as a source. We also extend this result to arbitrary fields under a mild assumption.
\end{abstract}


\end{frontmatter}


\section{Introduction}\label{s1}

Throughout this paper $p$ is a prime, and $\O$ is a complete discrete valuation ring with residue field $k$ of characteristic $p$. We allow the case $\O=k$.
Let $G$ and $H$ be finite groups.  An $\O G$-$\O H$-bimodule $M$ can be regarded as a left $\O (G\times H)$-module (and vice versa) via $(g,h)m=gmh^{-1}$, where $g\in G$, $h\in H$ and $m\in M$. If $M$ is indecomposable as an $\O G$-$\O H$-bimodule, then $M$ is indecomposable as an $\O(G\times H)$-module, hence has a vertex (in $G\times H$) and a source. If $\varphi:P\cong Q$ is an isomorphism between subgroups $P\leq G$ and $Q\leq H$, we set
$$\Delta\varphi:=\{(u,\varphi(u))\mid u\in P\},$$
and call it a {\it twisted diagonal subgroup} of $G\times H$; if $P=Q$ and $\varphi={\rm id}_P$, we denote $\Delta\varphi$ by $\Delta P$. We denote by $p_1:G\times H\to G$ and $p_2:G\times H\to H$ the canonical projections. It is easy to see that
a subgroup $X$ of $G\times H$ is twisted diagonal $\Longleftrightarrow$ $X\cong p_1(X)$ and $X\cong p_2(X)$. In this case we have $X=\{(u,p_2\circ p_1^{-1}(u))\mid u\in p_1(X)\}$, where we abusively use the same notation $p_1$ and $p_2$ to denote their restrictions to subgroups. In this paper we will frequently use the following context:

\begin{notation}\label{notation:stable equ of Morita type}
Let $G$ and $H$ be finite groups, $b$ a block of $\O G$ and $c$ a block of $\O H$. Assume that $M$ is an indecomposable
$(\O Gb, \O Hc)$-bimodule, which is finitely generated projective as a left and right module, inducing a stable equivalence of Morita type between $\O Gb$ and $\O Hc$. This means that $M\otimes_{\O Hc}M^*\cong \O Gb\oplus U_1$ for some projective $\O Gb\otimes_\O (\O Gb)^{\rm op}$-module $U_1$ and $M^*\otimes_{\O Gb} M\cong \O Hc\oplus U_2$ for some projective $\O Hc\otimes_\O (\O Hc)^{\rm op}$-module $U_2$, where $M^*:={\rm Hom}_\O(M,\O)$.
Let $X$ be a vertex of $M$ and $V$ an $\O X$-source of $M$.
\end{notation}

We give a new proof of the following result of Puig:

\begin{theorem}[{\cite[Corollary 7.4]{Puig1999}}]\label{theorem:vertices imply sources}
Keep the notation of \ref{notation:stable equ of Morita type}. Assume that $k$ is algebraically closed. Then the following are equivalent:
\begin{enumerate}[{\rm (i)}]
	\item $X$ is a twisted diagonal subgroup of $G\times H$.
	\item $X\cong p_1(X)$.
	\item $X\cong p_2(X)$.
	\item $V$ is an endopermutation $\O X$-module.
	\item $p$ does not divide the $\O$-rank of $V$.
\end{enumerate}
\end{theorem}

The implication (iv)$\Rightarrow$(v) is a property of endopermutation modules; see e.g. \cite[Proposition 7.3.10 (i)]{Lin18b}. The implication (v)$\Rightarrow$(i) follows from \cite[Proposition 5.11.8]{Lin18a}. 

\begin{remark}
{\rm Note that in both these two implications, the field $k$ can be arbitrary: Although \cite[Proposition 7.3.10 (i)]{Lin18b} needs the field $k$ to be perfect, using the classification of endopermutation modules over $k$ (\cite[Theorem 9.5]{Bouc:The Dade group} or \cite[Theorem 13.3]{Th07}), we see that any indecomposable capped endopermutation module is defined over a very small finite subfield. Hence the implication (iv)$\Rightarrow$(v) does not require the field $k$ to be perfect.}
\end{remark}

As mentioned above, (i) implies (ii) and (iii). Now it suffices to give a new proof of (ii)$\Rightarrow$(iv) because a symmetric argument will imply that (iii)$\Rightarrow$(iv) as well; this will be given in Section \ref{section:Proof of Puig's Corollary 7.4}. There are two main steps in our new proof of Theorem \ref{theorem:vertices imply sources}. The first step is Proposition \ref{prop:relation of source algebras}, which is the main difference between the methods of our proof and Puig's original proof. The proof of Proposition \ref{prop:relation of source algebras} is inspired by Linckelmann's proof of \cite[Theorem 9.11.9]{Lin18b}. The second main step is Puig's \cite[Theorem 7.2]{Puig1999} (see Theorem \ref{theorem: Puig 7.2} below). For this step, we follow the idea of Puig's original proof - we modify Puig's proof by simplifying terminology and notation, by changing some arguments and by adding details. We hope our modification can serve to explain Puig's proof of \cite[Theorem 7.2]{Puig1999}. Besides a new proof, we establish the result for arbitrary fields, under a mild assumption:

\begin{theorem}\label{theorem:descent to finite fields}
Keep the notation of \ref{notation:stable equ of Morita type}. Assume that $\O=k$. Assume that $X\cong p_1(X)$ or $X\cong p_2(X)$. Assume further that there are no subgroups $Z \unlhd Y$ of $X$ satisfying $Y/Z\cong Q_8$ (the quaternion group of order $8$) or that $k$ contains a primitive $3$-rd root of unity. Then $V$ is an endopermutation $kX$-module.
\end{theorem}

This will be proved in Section \ref{section:Descending to finite fields}, making use of extended versions of descent results for vertices and sources used by Kessar and Linckelmann \cite{Kessar_Linckelmann}, and the classification of indecomposable endopermutation modules. 

In Section \ref{section:Preliminaries on $G$-algebras and Brauer homomorphisms} we review some background terminology and results related to $G$-algebras and Brauer homomorphisms, most of which are in preparation for the proof of Theorem \ref{theorem: Puig 7.2}. The reader who is not concerned with the proof of Theorem \ref{theorem: Puig 7.2} only needs to read \ref{notation:basis notatin for algebras}, \ref{void: Brauer homomorphisms} and \ref{void:points} (iii) and skip the rest of Section \ref{section:Preliminaries on $G$-algebras and Brauer homomorphisms}. 

\section{Preliminaries on $G$-algebras and Brauer homomorphisms}\label{section:Preliminaries on $G$-algebras and Brauer homomorphisms}

\begin{notation}\label{notation:basis notatin for algebras}
{\rm  For any $\O$-algebra $A$, we denote by $A^\times$ the group of invertible elements of $A$ and by $A^{\rm op}$ the opposite $\O$-algebra of $A$. All $\O$-algebras in this paper are assumed to be finitely generated as $\O$-modules; this implies that all $k$-algebras in this paper are finite-dimensional. Unless otherwise specified, all modules are left modules. A homomorphism $f:A\to B$ between $\O$-algebras is not required to be unitary. Following Puig \cite{Puig pointed group}, we say that $f$ is an {\it embedding} if $\ker(f)=0$ and ${\rm Im}(f)=f(1_A)Bf(1_A)$. Following Puig \cite{Puig1988}, we say that $f$ is a {\it covering homomorphism} if $f(A)+J(B)=B$. Let $G$ and $H$ be finite groups. A {\it $G$-algebra over $\O$} is an $\O$-algebra $A$ endowed with an action of $G$ by $\O$-algebra automorphisms, denoted $a\mapsto {}^ga$, where $a\in A$ and $g\in G$. An {\it interior $G$-algebra over $\O$} is an $\O$-algebra $A$ with a group homomorphsim $G\to A^\times$, called the structure homomorphism. For any $g\in G$ we abusively use the same letter $g$ to denote the image of $g$ in $A^\times$. Let $A$ be an interior $G$-algebra over $\O$, $B$ an interior $H$-algebra over $\O$, and $V$ a left (resp. right) $\O H$-module. Then $\End_\O(V)$ is an interior $H$-algebra.  Let $f:G\to H$ be a group isomorphism. Then we denote by ${}_f V$ (resp. $V_f$) the left (resp. right) $\O G$-module, which is equal to $V$ as an $\O$-module, endowed with structure homomorphism $G\xrightarrow{f} H\to \End_\O(V)^\times$, and by ${}_f B$ the interior $G$-algebra with structure homomorphism $G\xrightarrow{f}H\to B^\times$. If $G=H$, we consider $A\otimes_\O B$ as an interior $G$-algebra over $\O$ with the structure homomorphism $G\cong \Delta G\hookrightarrow G\times G\to A^\times\times B^\times\to(A\otimes_\O B)^\times$. For a $k$-algebra $A$, we say that $A$ is {\it split}, if for any simple $A$-module $V$, we have $\End_A(V)\cong k$; we say that $A$ is {\it semisimple} if the Jacobson radical $J(A)=0$.
}	
\end{notation}

\begin{lemma}\label{lemma: Z(A) has stable basis}
Let $G$ be a finite group and $A$ a $G$-algebra over $k$. Assume that $A$ is semisimple and split. Then $Z(A)$ has a $G$-stable $k$-basis consisting of all primitive idempotents in $Z(A)$.
\end{lemma}

\begin{proof}
Since $G$ acts as $k$-algebra automorphisms on $A$, $G$ permutes idempotents in $Z(A)$. By the Wedderburn theorem (see e.g. \cite[Theorem 1.14.6]{Lin18a}), $A$ can be decomposed as $A= A_1\oplus \cdots \oplus A_n$, where $A_1,\cdots,A_n$ are isomorphic to matrix algebras over $k$. Let $e_1,\cdots,e_n$ be the identity elements of $A_1,\cdots,A_n$ respectively; they are exactly all primitive idempotents of $Z(A)$. Then $Z(A)= Z(A_1)\oplus\cdots\oplus Z(A_n)=k\cdot e_1\oplus\cdots\oplus k\cdot e_n$, whence the statement.
\end{proof}

\begin{lemma}\label{lemma:existence of h such that f=gh}
Let $G$ be a finite group, and let $A$ and $B$ be $\O$-algebras. Let $L$, $M$, $N$ and $U$ be (interior) $G$-algebras over $\O$ (resp. $A$-modules, or $A$-$B$-bimodules) with $U\subseteq N$. Let $f:M\to L$ and $g:N\to L$ be injective (interior) $G$-algebra (resp. $A$-modules, or $A$-$B$-bimodules) homomorphisms. Assume that ${\rm Im}(f)=g(U)$. There exists an injective (interior) $G$-algebra (resp. $A$-modules, or $A$-$B$-bimodules) homomorphism $h:M\to N$ such that $f=g\circ h$ and ${\rm Im}(h)=U$.
$$\xymatrix{ M\ar[rr]^{f}\ar@{-->}[d]_h \ar@{-->}[rrd]_h  & & L\\
U \ar@{^{(}->}[rr] &  &	N \ar[u]_g 
}$$
\end{lemma}	
\begin{proof}
Since $g$ is injective, the homomorphism $U\xrightarrow{g} g(U)$ is an isomorphism; we denote by $g^{-1}:g(U)\to U$ its inverse. Let $h:M\to N$ be the composition of the maps
$$M\xrightarrow{f} {\rm Im}(f)=g(U)\xrightarrow{g^{-1}} U\hookrightarrow N.$$
Then it is clear that $h$ is injective, ${\rm Im}(h)=U$ and we have $f=g\circ h$.
\end{proof}

\begin{proposition}\label{prop: matrix algebras unique embedding}
Let $A=M_m(\O)$ and $B=M_n(\O)$ be matrix algebras over $\O$ with $m\leq n$. Then there exists an embedding $f:A\to B$ of $\O$-algebras. If $g:A\to B$ is another embedding of $\O$-algebras, then there is an element $b\in B^\times$ such that $f=c_b\circ g$, where $c_b$ is the inner automorphism of $B$ induced by $b$-conjugation. 
\end{proposition}	
	
\begin{proof} In this proof, for any $\O$-module $V$ and any $v\in V$, we denote by $\bar{v}$ the image of $v$ in $\bar{V}:=V/J(\O) V$; if $V$ is $\O$-free, we denote by $\rk_\O(V)$ its $\O$-rank. Note that any $\O$-submodule of a finitely generated free $\O$-module is $\O$-free; see e.g. \cite[Proposition 1.5]{Thevenaz}.

Let $f:A\to B$ be the map sending any matrix $M\in A$ to the block diagonal matrix  $\left({\begin{array}{*{20}{c}}M& \\
	& 0 \end{array}}\right)\in B$. One easily checks that $f$ is an embedding of $\O$-algebras, proving the first statement. If $g:A\to B$ is another embedding of $\O$-algebras, we have 
	$$\rk_\O(f(1_A)Bf(1_A))=\rk_\O({\rm Im}(f))=\rk_\O(A)=\rk_\O({\rm Im}(g))=\rk_\O(g(1_A)Bg(1_A)).$$
	By elementary linear algebra, for any idempotent $e\in B=M_n(\O)$, 
	$$\rk_\O(eBe)=\dim_k(\bar{e}\bar{B}\bar{e})={\rm rank}(\bar{e}),$$ where ${\rm rank}(\bar{e})$ is the rank of the matrix $\bar{e}$. Hence we deduce that ${\rm rank}(\overline{f(1_A)})={\rm rank}(\overline{g(1_A)})$.  It follows, again by elementary linear algebra, that $\overline{f(1_A)}$ and $\overline{g(1_A)}$ are conjugate in $\bar{B}$. By lifting theorem of idempotents (see e.g. \cite[Theorem 3.1 (c)]{Thevenaz}), there exists $b'\in B^\times$ such that $f(1_A)=b'g(1_A)b'^{-1}$. Now we have
	$$ {\rm Im}(f)=f(1_A)Bf(1_A)=b'g(1_A)Bg(1_A)b'^{-1}=b'{\rm Im}(g)b'^{-1}={\rm Im}(c_{b'}\circ g). $$
Now by Lemma \ref{lemma:existence of h such that f=gh}, there exists an automorphism $h$ of $A$ such that $f=c_{b'}\circ g\circ h$. By the Skolem--Noether theorem (see e.g. \cite[Theorem 2.8.12]{Lin18a}), $h=c_a$ for some $a\in A^\times$. Let $a'=g(a)+1_B-g(1_A)$. Then $a'\in B^\times$ with inverse $a'^{-1}=g(a^{-1})+1_B-g(1_A)$. One easily checks that $g\circ h=g\circ c_a=c_{a'}\circ g$. So writing $b=b'a'\in B^\times$, then $f=c_{b'}\circ c_{a'}\circ h=c_b\circ g$.
\end{proof}

\begin{void}\label{void: Brauer homomorphisms}
	{\rm \textbf{Brauer homomorphisms.}	Let $G$ be a finite group, and let $A$ and $B$ be $G$-algebras over $\O$ (resp. $\O G$-modules). For any subgroup $P$ of $G$, we denote by $A^P$ the $N_G(P)$-subalgebra (resp. $kN_G(P)$-submodule) of $P$-fixed points of $A$.  For any two $p$-subgroups $Q\leq P$ of $G$, the {\it relative trace map} ${\rm Tr}_Q^P:A^Q\to A^P$ is defined by ${\rm Tr}_Q^P(a)=\sum_{x\in [P/Q]}{}^xa$, where $[P/Q]$ denotes a set of representatives of the left cosets of $Q$ in $P$. We denote by $A(P)$ the {\it $P$-Brauer quotient} of $A$, i.e., the $N_G(P)$-algebra (resp. $kN_G(P)$-module)
		\[A^P/(\sum_{Q<P}{\rm Tr}_Q^P(A^Q)+J(\O)A^P).\]
		We denote by ${\rm br}_P^A:A^P\to A(P)$ the canonical map, which is called the {\it $P$-Brauer homomorphism}. Sometimes we write $\br_P$ instead of $\br_P^A$ if no confusion arises. The following properties of Brauer homomorphisms are well-known and easy to check:
		
	(i)	If $f:A\to B$ is a homomorphism of $G$-algebras (resp. $\O G$-modules), then $f$ restricts to a homomorphism of $N_G(P)$-algebras (resp. $\O N_G(P)$-modules) $f^P:A^P\to B^P$, which in turn induces a homomorphism of $N_G(P)$-algebras (resp. $kN_G(P)$-modules) 
	\begin{equation}\label{equation: f(P)}
	f(P): A(P)\to B(P)
	\end{equation}
	 sending $\br_P^A(a)$ to $\br_P^B(f(a))$ for any $a\in A^P$. In other words, we have $f(P)\circ \br_P^A=\br_P^B\circ f^P$. In this way, the $P$-Brauer construction defines a functor (called the $P$-Brauer functor) from the category of $G$-algebras (resp. $\O G$-modules) to the category of $N_G(P)$-algebras (resp. $kN_G(P)$-modules). 
		
(ii) If $Q$ is a $p$-subgroup of $N_G(P)$ containing $P$, then the $P$-Brauer homomorphism $\br_P^A:A^P\to A(P)$ restricts to a homomorphism $(\br_P^A)^Q: A^Q\to A(P)^Q$ of $N_G(P,Q)$-algebras (resp. $kN_G(P,Q)$-modules) (where $N_G(P,Q):=N_G(P)\cap N_G(Q)$), which in turn induces a homomorphism of $N_G(P,Q)$-algebras (resp. $kN_G(P,Q)$-modules)
		\begin{equation}\label{equation: Brauer const. transitivity}
		\alpha_A(P,Q):A(Q)\to A(P)(Q)
		\end{equation}
		sending $\br_Q^A(a)$ to $\br_Q^{A(P)}(\br_P^A(a))$ for any $a\in A^Q$. Note that $\alpha_A(P,Q)$ is exactly $(\br_P^A)(Q)$ in the sense of (i).

(iii) The inclusion map $A^P\otimes_\O B^P\to (A\otimes_\O B)^P$ induces a homomorphism of $N_G(P)$-algebras (resp. $kN_G(P)$-modules)
		\begin{equation}\label{equation: Brauer map tensor products}
		\alpha_{A,B}(P):A(P)\otimes_k B(P)\to (A\otimes_\O B)(P)
		\end{equation}
		sending $\br_P(a)\otimes \br_P(b)$ to $\br_P(a\otimes b)$ for any $a\in A^P$ and $b\in B^P$. In the algebra case, it is clear that both $\alpha_A(P,Q)$ and $\alpha_{A,B}(P)$ are unitary homomorphisms.		 
	}
\end{void}

\begin{void}\label{void:covering homomorphisms}
{\rm \textbf{Covering homomorphisms of $G$-algebras.} Let $G$ be a finite group and $f:A\to B$ a homomorphism of $G$-algebras over $\O$. Following Puig \cite{Puig pointed group}, we say that $f$ is an {\it embeding of $G$-algebras} if the underlying $\O$-algebra homomorphism is an embedding. Following Puig \cite{Puig1988}, we say the $f$ is a {\it covering homomorphism of $G$-algebras} if for every subgroup $H$ of $G$, the $\O$-algebra homomorphism $f^H:A^H\to B^H$ is a covering homomorphsim (which means that $f^H(A^H)+J(B^H)=B^H$). According to \cite[Theorem 4.22]{Puig1988} (see \cite[Theorem 25.9]{Thevenaz}), $f$ is a covering homomorphism of $G$-algebras if and only if for every $p$-subgroup of $G$, $f(P):A(P)\to B(P)$ is a covering homomorphism of $k$-algebras. 
}
\end{void}

\begin{lemma}\label{lemma:embbeding}
Let $G$ be a finite group and let $A$ and $B$ be $G$-algebras over $\O$. Let $f:A\to B$ be an embedding of $G$-algebras. For any $p$-subgroup $P$ of $G$, the following hold:
\begin{enumerate}[{\rm (i)}]
	\item  The restriction $f^P:A^P\to B^P$ of $f$ is an embedding.
	\item  The map $f(P):A(P)\to B(P)$ in (\ref{equation: f(P)}) is an embedding.
	\item  If $i\in A^P$ is a primitive (local) idempotent, then $f(i)$ is a primitive (local) idempotent.
\end{enumerate}

\end{lemma}

\begin{proof}
Since $f$ is an embedding, we have $\ker(f)=0$ and ${\rm Im}(f)=f(1_A)Bf(1_A)$. Hence $\ker(f^P)=0$. Clearly we have ${\rm Im}(f^P)=f(A^P)\subseteq f(1_A)B^Pf(1_A)$. Assume that $b\in f(1_A)B^Pf(1_A)$. Then $b\in {\rm Im}(f)$, hence $b=f(a)$ for some $a\in A$. For any $u\in P$, we have $f({}^ua)={}^u(f(a))={}^ub=b$. Since $f$ is injective, this implies ${}^ua=a$, i.e., $a\in A^P$. So $f(A^P)= f(1_A)B^Pf(1_A)$, proving (i).

According to the diagram $$\xymatrix{A^P \ar[rr]^{f^P}\ar[d]_{\br_P^A} &  &   B^P\ar[d]^{\br_P^B} \\
A(P)\ar[rr]^{f(P)}  &  & B(P) }$$
we have 
$${\rm Im}(f(P))=\br_P^B(f^P(A^P))=\br_P^B(f(1_A)B^Pf(1_A))=\br_P^B(f(1_A))\br_P^B(B^P)\br_P^B(f(1_A))$$ $$=f(P)(1_{A(P)})B(P)f(P)(1_{A(P)}),$$
where the second equality holds by (i). In order to show that $\ker(f(P))=0$, it suffices to show that for any $a\in A^P$ with  $f(P)(\br_P^A(a))=0$, we have $a\in \sum_{Q<P}{\rm Tr}_Q^P(A^Q)+J(\O)A^P$. Since $f(P)(\br_P^A(a))=\br_P^B(f(a))=0$, we have $f(a)\in \sum_{Q<P}{\rm Tr}_Q^P(B^Q)+J(\O)B^P$. Since $f(a)=f(1_A)f(a)f(1_A),$ we have
\begin{align*}
\begin{split}
	f(a) &\in \sum_{Q<P}{\rm Tr}_Q^P(f(1_A)B^Qf(1_A))+J(\O)f(1_A)B^Pf(1_A)\\
	&=\sum_{Q<P}{\rm Tr}_Q^P(f(A^Q))+J(\O)f(A^P)=f(\sum_{Q<P}{\rm Tr}_Q^P(A^Q)+J(\O)A^P),
\end{split}
\end{align*}
where the first equality holds by (i). Since $f$ is injective, this imlies that $a\in \sum_{Q<P}{\rm Tr}_Q^P(A^Q)+J(\O)A^P$, completing the proof of (ii). Statement (iii) follows from \cite[Propositions 8.5 and 15.1 (d)]{Thevenaz}.
\end{proof}	

\begin{lemma}[cf. e.g. {\cite[Proposition 5.8.1]{Lin18a}}]\label{lemma:basis}
Let $G$ be a finite group and $P$ a $p$-subgroup of $G$. Let $A$ be a $G$-algebra over $\O$ (resp. an $\O G$-module). If $A$ has a $P$-stable $\O$-basis $X$, then $\{\br_P^A(a)\mid a\in X^P\}$ is a $k$-basis of $A(P)$, where $X^P:=\{x\in X\mid {}^ux=x,~\forall~u\in P\}$.
\end{lemma}

\begin{lemma}[{\cite[Lemma 7.10]{Puig1999}}]\label{lemma:Puig 7.10}
Let $G$ be a finite group and let $A$ and $B$ be $G$-algebras over $\O$ (resp. $\O G$-modules). If $A$ has a $P$-stable $\O$-basis, then for any $p$-subgroup $P$ of $G$ and any $p$-subgroup $Q$ of $N_G(P)$ containing $P$, both $\alpha_A(P,Q)$ and $\alpha_{A,B}(P)$ are isomorphisms.  
\end{lemma}

\begin{proof}
It suffices to show that $\alpha_A(P,Q)$ and $\alpha_{A,B}(P)$ are isomorphisms of $k$-modules. The proof of \cite[Lemma 7.10]{Puig1999} is already a detailed proof.  Alternatively, refer to \cite[Proposition 5.8.10]{Lin18a} for a proof of the fact that $\alpha_{A,B}(P)$ is an isomorphism of $k$-modules. By Lemma \ref{lemma:basis}, one easily sees that $\alpha_A(P,Q)$ is an isomorphism of $k$-modules.
\end{proof}

In the next lemma, the first three statements generalise \cite[Proposition 5.4.6]{Lin18a}, and the last statement is \cite[Lemma 1.11]{BP}.
\begin{lemma}\label{lemma:intersection of kernels of Brauer homomorphisms}
Let $G$ be a finite group, $A$ a $G$-algebra over $\O$ (resp. $\O G$-module). Let $P$ be a $p$-subgroup of $G$. Assume that $A$ has a $G$-stable $\O$-basis $X$. Let $B$ be the $N_G(P)$-subalgebra (resp. $\O N_G(P)$-submodule) of $A^P$ generated by $X^P:=\{x\in X\mid {}^ux=x,~\forall~u\in P\}$. The following hold:
\begin{enumerate}[{\rm (i)}]
	\item $A_P^G=B_P^G+\sum_{Q<P}A_Q^G$.
	\item Assume that $\O=k$.  By Lemma \ref{lemma:basis}, $\br_P^A$ restricts to an isomorphism $B\cong A(P)$ of $k$-modules, sending $x\in X^P$ to $\br_P^A(x)$. Denote by $f$ the inverse of this isomorphism. Then the map $\Tr_{N_G(P)}^G\circ f$ induces a $k$-linear isomorphism $A(P)_P^{N_G(P)}\cong B_P^G$ with inverse map induced by $\br_P^A$.
	\item $\ker(\br_P^A)\cap A_P^G=\sum_{Q<P} A_Q^G+J(\O)A_P^G$.
	\item  $\bigcap_{1<Q\leq P} \ker(\br_Q^A)=A_1^P+J(\O)A^P$. 
\end{enumerate}

\begin{proof}
For the purpose of this lemma, we may assume that $\O=k$.

(i) The right side of statement (i) is trivially contained in the left side. For the reverse inclusion, let $P_x$ be the stabiliser of $x$ in $P$ for any $x\in X$. Then $A^P$ has a $k$-basis $\{\Tr_{P_x}^P(x)\mid x\in X\}$. It follows that as a $k$-module, $A_P^G$ is generated by the set $\{\Tr_{P_x}^G(x)\mid x\in X\}$. If $P_x=P$, then $\Tr_{P_x}^G(x)\in B_P^G$. If $P_x<P$, then $\Tr_{P_x}^G(x)\in \sum_{Q<P}A_Q^G$, proving (i). 

(ii) We will abusively use the same notation $\br_P^A$ to denote its restrictions to subalgebras (resp. submodules). We see that $\br_P^A:B\to A(P)$ is not only an isomorphism of $k$-modules, but also an isomorphism of $N_G(P)$-algebras (resp. $kN_G(P)$-modules).  Also, its inverse $f$ is an isomorphism of $N_G(P)$-algebras (resp. $kN_G(P)$-modules). As a $k$-module, $B_P^G$ is generated by the set $\{\Tr_P^G(x)\mid x\in X^P\}$. For any $x\in X^P$, we have 
\begin{align*}
	\begin{split}
\Tr_{N_G(P)}^G\circ f\circ\br_P^A(\Tr_P^G(x)) &=\Tr_{N_G(P)}^G\circ f(\Tr_P^{N_G(P)}(\br_P^A(x)))\\
&=\Tr_{N_G(P)}^G\circ \Tr_P^{N_G(P)}\circ f\circ \br_P^A(x)=\Tr_P^G(x);
\end{split}
\end{align*}
see e.g. \cite[Proposition 5.4.5]{Lin18a} for the first equality. This implies that $\Tr_{N_G(P)}^G\circ f\circ\br_P^A$ is the identity map on ${B_P^G}$. On the other hand, as a $k$-module, $A(P)_P^{N_G(P)}$ is generated by $\{\Tr_P^{N_G(P)}(\br_P^A(x))\mid x\in X^P\}$. For any $x\in X^P$, we have 
\begin{align*}
	\begin{split}
	\br_P^A\circ\Tr_{N_G(P)}^G\circ f (\Tr_P^{N_G(P)}(\br_P^A(x)))&=\br_P^A\circ\Tr_{N_G(P)}^G\circ  \Tr_P^{N_G(P)}\circ f (\br_P^A(x))\\
	&=\br_P^A(\Tr_P^G(x))=\Tr_P^{N_G(P)}(\br_P^A(x)).
	\end{split}
\end{align*}
Hence $\br_P^A\circ\Tr_{N_G(P)}^G\circ f$ equals the identify map on ${A(P)_P^{N_G(P)}}$, proving statement (ii).

(iii) Using (i), it suffices to show that $\ker(\br_P^A)\cap B_P^G=0$. Indeed, if $a\in \ker(\br_P^A)\cap B_P^G$, then by (ii) we have $a=\Tr_{N_G(P)}^G\circ f\circ \br_P^A(a)=0$. 

(iv) The right side of statement (iv) is contained in the left side. Indeed, for any nontrivial subgroup $Q$ of $P$ and any $a\in A$, by Mackey's formula, one has $$\br_Q^A(\Tr_1^P(a))=\br_Q^A(\sum_{t\in [Q\backslash G/1]}\Tr_1^Q({}^ta))=0.$$
 Using repeatedly statement (iii), we obtain the reverse inclusion.
\end{proof}

\end{lemma}

\begin{lemma}[{\cite[Lemmas 7.11--7.14]{Puig1999}}]\label{lemma: commutative diagrams}
Let $G$ be a finite group and let $A$, $B$, $C$ and $D$ be $G$-algebras over $\O$ (resp. $\O G$-modules). Let $f:A\to B$ and $g:C\to D$ be a homomorphism of $G$-algebras (resp. $\O G$-modules). Then for any $p$-subgroup $P$ of $G$, any $p$-subgroup $Q$ of $N_G(P)$ containing $P$, and any $p$-subgroup $R$ of $N_G(P,Q)=N_G(P)\cap N_G(Q)$ containing $Q$, the following diagrams are commutative:
\begin{equation}
\xymatrix{ A(P)(Q) \ar[rr]^{f(P)(Q)} &  &  B(P)(Q)\\
	A(Q)  \ar[rr]^{f(Q)} \ar[u]^{\alpha_A(P,Q)} &  & B(Q)\ar[u]_{\alpha_B(P,Q)}}~~~\xymatrix{(A\otimes_\O C)(P) \ar[rr]^{(f\otimes g)(P)} &  & (B\otimes_\O D)(P) \\ 
	A(P)\otimes_k C(P)\ar[rr]^{f(P)\otimes g(P)}  \ar[u]^{\alpha_{A,C}(P)}  & &  B(P)\otimes_k D(P)\ar[u]_{\alpha_{B,D}(P)}   }  \tag{i}
	\end{equation}
\begin{equation}
\xymatrix{(A\otimes_\O B)(Q) \ar[rr]^{\alpha_{A\otimes_\O B}(P,Q)}  &   &  (A\otimes_\O B)(P)(Q) \\
  &    &   (A(P)\otimes_k B(P))(Q)\ar[u]_{(\alpha_{A,B}(P))(Q)}\\
  A(Q)\otimes_k B(Q)\ar[uu]^{\alpha_{A,B}(Q)}\ar[rr]^{\alpha_A(P,Q)\otimes \alpha_B(P,Q)}  &  ~~~~~~~~~~  &  A(P)(Q)\otimes_k B(P)(Q)\ar[u]_{\alpha_{A(P),B(P)}(Q)}} \tag{ii}
  \end{equation}
\begin{equation}
\xymatrix{A(Q)(R)\ar[rr]^{(\alpha_A(P,Q))(R)~~} &  & A(P)(Q)(R)\\
A(R)\ar[u]^{\alpha_A(Q,R)}   \ar[rr]^{\alpha_A(P,R)} &  &  A(P)(R)\ar[u]_{\alpha_{A(P)}(Q,R)}}  \tag{iii}
\end{equation}
\begin{equation}
\xymatrix{ (A\otimes_\O B)(P)\otimes_k C(P)\ar[rr]^{\alpha_{A\otimes_\O B, C}(P)} &  &  (A\otimes_\O B\otimes_\O C)(P)\\
A(P)\otimes_kB(P)\otimes_k C(P)\ar[u]^{\alpha_{A,B}(P)\otimes {\rm id}_{C(P)}} \ar[rr]^{~{\rm id}_{A(P)}\otimes \alpha_{B,C}(P)}  &    &   A(P)\otimes_k (B\otimes_\O C)(P)\ar[u]_{\alpha_{A,B\otimes_\O C}(P)}}   \tag{iv}
\end{equation}

\end{lemma}
\begin{proof}
The verification of the commutative diagrams is straightforward by using the definitions of the homomorphisms in (\ref{equation: f(P)}), (\ref{equation: Brauer const. transitivity}) and (\ref{equation: Brauer map tensor products}). The verifications in the proofs of \cite[Lemmas 7.11--7.14]{Puig1999} are already detailed.
\end{proof}

\begin{lemma}[{\cite[Lemma 7.16]{Puig1999}}]\label{lemma: linear form Puig}
Let $P$ be a finite $p$-group and $A$ an interior $P$-algebra over $\O$ having a $P$-stable $\O$-basis. Assume there is a nondegenerate symmetric $\O$-linear form $s:A\to \O$. This means that $s\in \Hom_\O(A,\O)$ satisfying $s(ab)=s(ba)$ for any $a,b\in A$, and for any element $a\in A$, there exists $b\in A$ such that $s(ab)\neq 0$. Assume that $A(P)\neq 0$. Then for any subgroup $Q$ of $P$, $s$ induces a nondegenerate symmetric $k$-linear form $s(Q):A(Q)\to k$, sending $\br_Q^A(a)$ to $\overline{s(a)}$ for any $a\in A^Q$, where $\overline{s(a)}$ is the image of $s(a)$ in $k$.
\end{lemma}

\begin{proof}
	For the purpose of this lemma we may assume that $\O=k$. Let $X$ be a $P$-stable $k$-basis of $A$. We regard $A$ as a $k P$-module via the $P$-conjugation action. We regard $k$ as a trivial $k P$-module. Since $s:A\to k$ is symmetric, we have $s(xax^{-1})=s(x^{-1}xa)=s(a)=xs(a)$ for any $x\in P$ and $a\in A$, where the last equality holds because $k$ is a trivial $kP$-module. This implies that $s$ is a homomorphism of $kP$-modules. Hence we can apply the $Q$-Brauer functor to $s$ and obtain a homomorphism $s(Q):A(Q)\to k$ of $kN_G(Q)$-modules; see (\ref{equation: f(P)}). By the construction of $s(Q)$, we have 
	$$s(Q)(\br_Q(a)\br_Q(b))=s(Q)(\br_Q(ab))=s(ab)=s(ba)=s(Q)(\br_Q(b)\br_Q(a))$$
	for any $a,b\in A^Q$. Hence the $k$-linear map $s(Q)$ is symmetric. Let $Y:=\{y_x\mid x\in X\}\subseteq A$ be a dual $k$-basis with respect the symmetric form $s$; that is $y_x$ is the unique element of $A$ such that $s(xy_x)=1$ and $s(x'y_x)=0$ for any $x\neq x'\in X$. It is easy to check that $Y$ is another $P$-stable $k$-basis of $A$, and $Y^Q=\{y_x\mid x\in X^Q\}$. Hence $\br_Q(X^Q)$ and $\br_Q(Y^Q)$ are a pair of dual $k$-basis of $A(Q)$ with respect to the $k$-form $s(Q)$, which implies that $s(Q)$ is nondegenerate.
\end{proof}

\begin{void}\label{void:points}
	
	{\rm (i) Let $G$ be a finite group and $A$ an $\O$-algebra. Following Puig \cite{Puig pointed group}, a {\it point} of $A$ is an $A^\times$-conjugacy class of a primitive idempotent in $A$. Let $I$ be a primitive decomposition of $1_A$ in $A$. The {\it multiplicity} of a point $\alpha$ on $A$ is the cardinal $m_\alpha=|I\cap \alpha|$ of the set $I\cap \alpha$, and it does not depend on the choice of $I$. Denote by $\P(A)$ the set of points of $A$. The map sending a primitive idempotent $i\in A$ to the projective (resp. simple) $A$-module $Ai$ (resp. $Ai/J(A)i$) induces a bijection between $\P(A)$ and the set of  isomorphism classes of projective (resp. simple) modules; see e.g. \cite[Proposition 4.7.17 (i),(ii)]{Lin18a}. Suppose that $A$ is split. Then for any point $\alpha$ of $A$ and any $i\in \alpha$, we have $\dim_k(Ai/J(A)i)=m_\alpha$, and hence $\End_k(Ai/J(A)i)$ is isomorphic to the matrix algebra $M_{m_\alpha}(k)$; see e.g. \cite[Proposition 4.7.17 (iv)]{Lin18a}. Denote the $k$-algebra $\End_k(Ai/J(A)i)$ by $A(\alpha)$. It is clear that up to isomorphism $A(\alpha)$ is independent of the choice of $i$. By Wedderburn's theorem for split algebras (see \cite[Theorem 1.14.6]{Lin18a}), the sum of the structure homomorphisms $A\to \End_k(Ai/J(A)i)=A(\alpha)$ induces an isomorphism of $k$-algebras
		\begin{equation}\label{equation:direct product decomposition}
		A/J(A)\cong \oplus_{\alpha\in \P(A)}A(\alpha).
		\end{equation}
		
(ii)	Let $A$ be a $G$-algebra over $\O$ and $P$ a subgroup of $G$. A {\it point of $P$ on $A$} is a point $\alpha$ of $A^P$.  Following Puig \cite{Puig pointed group}, we call the pair $P_\alpha$ a pointed group on $A$. Let $Q_\beta$ be another pointed group on $A$. We say that $Q_\beta$ is contained in $P_\alpha$ and write $Q_\beta\leq P_\alpha$ if $Q\leq P$ and if there are $i\in \alpha$ and $j\in \beta$ such that $ij=j=ji$.    If ${\rm br}_P^A(\alpha)\neq \{0\}$, the point $\alpha$ is called a {\it local point of $P$ on $A$}, any idempotent in $\alpha$ is called a {\it primitive local idempotent of $P$ on $A$}, and $P_\alpha$ is called a {\it local pointed group}; we know that in this case ${\rm br}_P^A(\alpha)$ is a point of $A(P)$. We use the symbol $\LP_A(P)$ to denote the set of local points of $P$ on $A$. The correspondence $\alpha\mapsto {\rm br}_P^A(\alpha)$ induces a bijection between $\LP_A(P)$ and $\P(A(P))$; see e.g. \cite[Lemma 14.5]{Thevenaz}.

(iii) Let $G$ be a finite group. A {\it block} of $\O G$ is a primitive idempotent $b$ in $Z(\O G)$, and $\O Gb$ is called a {\it block algebra} of $\O G$. A defect group of $b$ is a maximal subgroup $P$ of $G$ such that $\br_P^{\O G}(b)\neq 0$. A primitive idempotent $i\in (\O Gb)^P$ is called a {\it source idempotent} of $b$ if $\br_P^{\O G}(i)\neq 0$, and the interior $P$-algebra $i\O Gi$ is called a {\it source algebra} of $b$; see \cite[\S3]{Puig pointed group} or \cite[Definition 6.4.1]{Lin18b}. It is well-known that block algebras and source algebras are symmetric algebras; see e.g. \cite[Theorem 2.11.11]{Lin18a}. By \cite[3.5]{Puig pointed group} or \cite[Theorem 6.4.6]{Lin18b}, the $\O Gb$-$i\O Gi$-bimodule $\O Gi$ and its dual $i\O G$ induce a Morita equivalence between $\O Gb$ and $i\O Gi$. 
}
\end{void}

\begin{lemma}\label{lemma: primitive idempotent trace 1}
Let $A$ be a $k$-algebra such that $A/J(A)$ is split. Let $V$ be a simple $A$-module and $\eta:A\to \End_k(V)$ the structure homomorphism. For any primitive idempotent $i\in A$, the trace of the linear transformation $\eta(i)$ is $1$. 
\end{lemma}
\begin{proof}
Since the structure homomorphism $\eta$ factors through $A/J(A)$ and since the image of $i$ in $A/J(A)$ is still a primitive idempotent (see e.g. \cite[Corollary 4.7.8]{Lin18a}), we may assume that $J(A)=0$. By Wedderburn's theorem (see e.g. \cite[Theorem 1.14.6]{Lin18a}), $A$ is isomorphic to a direct sum of finitely many matrix algebras over $k$. Now it is easy to see that $i$ is contained in one of those matrix algebras (one can also use Rosenberg's lemma, see. e.g. \cite[Corollary 4.4.8]{Lin18a}, for this). So we may assume that $A$ is a matrix algebra over $k$. Now the statement is a well-known fact of elementary linear algebra.
\end{proof}



\begin{lemma}[see {\cite[page 267]{Puig pointed group}}]\label{lemma:local point isomorphism}
Let $P$ be a finite $p$-group, and $A$ a $P$-algebra over $\O$. Let $\alpha$ be a local point of $P$ on $A$. By \ref{void:points} (ii), $\br_P^A(\alpha)$ is a point of $A(P)$. Let $m_\alpha$ be the multiplicity of $\alpha$ on $A^P$ and $m_{\br_P^A(\alpha)}$ be the multiplicity of $\br_P^A(\alpha)$ on $A(P)$. Then $m_\alpha=m_{\br_P^A(\alpha)}$.  If $A^P$ is split, then $A^P(\alpha)\cong M_{m_\alpha}(k)\cong A(P)(\br_P^A(\alpha))$.
\end{lemma}
\begin{proof}
By \cite[Theorem 4.7.19]{Lin18a}, we have $m_\alpha=m_{\br_P^A(\alpha)}$.	
If $A^P$ is split, since $A(P)$ is a quotient of $A^P$, $A(P)$ is split as well. By \ref{void:points} (i), $A^P(\alpha)$ and $A(P)(\br_P^A(\alpha))$ are matrix algebras over $k$ of dimensions $m_\alpha^2$ and $m_{\br_P^A(\alpha)}^2$ respectively, proving the statement. 
\end{proof}

\begin{proposition}\label{prop: unique P stable point}
Let $A$ be a $P$-algebra over $k$ which is split and semisimple as a $k$-algebra. By (\ref{equation:direct product decomposition}), we have $A\cong \oplus_{\alpha'\in \P(A)}A(\alpha')$. Identify these two algebras via the isomorphism. Assume that $A(P)$ has a unique point. 
\begin{enumerate}[{\rm (i)}]
	\item Then there is a unique $P$-stable point $\alpha\in\P(A)$ such that $A(P)\cong (A(\alpha))(P)$. 
	\item Let $e_\alpha$ be the identity element of $A(\alpha)$. Then $1-e_\alpha \in \ker(\br_P^{Z(A)})$.
\end{enumerate}
\end{proposition}
\begin{proof}
For any $\alpha'\in \P(A)$, denote by $e_{\alpha'}$ the unity element of $A(\alpha')$; then we have $e_{\alpha'}\in Z(A)$ and $A(\alpha')=Ae_{\alpha'}$. Since $P$ acts as $k$-algebra automorphisms on $A$, the group $P$ permutes the set $\P(A)$ and hence permutes the set $\{e_{\alpha'}\mid \alpha'\in \P(A)\}$. Let  ${\rm stab}_P(\alpha')$ be the stabiliser of $\alpha'$ in $P$. If ${\rm stab}_P(\alpha')<P$, then $\br_P({\rm Tr}_{{\rm stab}_P(\alpha')}^P(e_{\alpha'}))=0$, and hence $\br_P^{A}(\sum_{u\in [P/\stab(\alpha)]}A{}^ue_\alpha)=0$. Since $A(P)$ has a uniqiue point, we deduce that there exists a unique $\alpha\in \P(A)$ such that $\stab_P(\alpha)=P$ and $A(P)=(Ae_\alpha)(P)$. Otherwise $A(P)$ would be a direct sum of two $k$-algebras, which contradicts the assumption that $A(P)$ has a unique point. This proves (i).

By Lemma \ref{lemma: Z(A) has stable basis}, the set $X:=\{e_{\alpha'}\mid \alpha'\in \P(A)\}$ is a $P$-stable $k$-basis of $Z(A)$. By (i), we have $X^P=\{e_\alpha\}$. Now statement (ii) is clear from Lemma \ref{lemma:basis}.
\end{proof}

\begin{proposition}\label{prop:bijection between points}
Let $A$ and $B$ be $\O$-algebras. Assume that at least one of $A$ or $B$ is split. Then there is a bijection $\P(A)\times \P(B)\to \P(A\otimes_\O B)$ sending $(\alpha,\beta)$ to $\alpha\times \beta$, where $\alpha\times \beta$ is the point of $A\otimes_\O B$ containing the set $\{i\otimes j\mid i\in\alpha,j\in \beta\}$.  
\end{proposition}

\begin{proof}
For the purpose of this proposition, by the lifting theorem of idempotents (see e.g. \cite[Theorem 4.7.1]{Lin18a}), we may assume that $\O=k$. Let $\alpha\in\P(A)$ and $\beta\in\P(B)$. For any $i,i'\in\beta$ and $j,j'\in \beta$, $i\otimes j$ is a primitive idempotent in $A\otimes_k B$ (see e.g. \cite[Lemma 2.2]{KL}), and $i\otimes j$, $i'\otimes j'$ are conjugate in $A\otimes_\O B$. Hence the map $(\alpha,\beta)\mapsto \alpha\times \beta$ is well-defined.  Since at least one of $A/J(A)$ or $B/J(B)$ is split semisimple, by \cite[Proposition 1.16.19]{Lin18a}, at least one of them is separable. Then by \cite[Corollary 1.16.15]{Lin18a}, $J(A\otimes_k B)=J(A)\otimes_k B+A\otimes_k J(B)$. Hence we have
$$(A\otimes_k B)/J(A\otimes_k B)\cong A/J(A)\otimes_k B/J(B).$$ 
Again by the lifting theorem of idempotents, there are bijections $\P(A)\to P(A/J(A))$, $\P(B)\to \P(B/J(B))$ and $\P(A\otimes_k B)\to \P((A\otimes_k B)/J(A\otimes_k B))$, hence we may assume that $J(A)=0$ and $J(B)=0$. Now for any $\alpha\in\P(A)$, $\beta\in\P(B)$, $i\in \alpha$ and $j\in\beta$, $Ai$ is a simple $A$-module and $Bj$ is a simple $B$-module; see \ref{void:points} (i). Hence by \cite[Theorem 10.38 (ii)]{Curtis-Reiner}, $Ai\otimes_k Bj\cong (A\otimes_k B)(i\otimes j)$ is a simple $A\otimes_k B$-module. On the other hand, by \cite[Theorem 10.38 (iii)]{Curtis-Reiner}, any simple $A\otimes_k B$-module $S$ is of the form $S_1\otimes_k S_2$, where $S_1$ is a simple $A$-module and $S_2$ is a simple $B$-module. Hence $S$ is of the form $Ai\otimes_k Bj\cong (A\otimes_k B)(i\otimes j)$ for some $\alpha\in \P(A)$, $\beta\in\P(B)$, $i\in\alpha$ and $j\in \beta$; see \ref{void:points} (i). Moreover, according to \cite[Theorem 10.38 (iii)]{Curtis-Reiner}, for any $\alpha,\alpha'\in \P(A)$ and $\beta,\beta'\in \P(B)$ and any $i\in \alpha$, $i'\in \alpha'$, $j\in \beta$ and $j'\in \beta'$, $Ai\otimes_k Bj\cong Ai'\otimes_k Bj'$ if and only if $Ai\cong Ai'$ and $Bj\cong Bj'$. This is in turn equivalent to $(\alpha,\beta)=(\alpha',\beta')$; see \ref{void:points} (i). In conclusion, the map $(\alpha,\beta)\mapsto \alpha\times \beta$ is a bijection. (Note that in \cite[Theorem 10.38]{Curtis-Reiner}, it is assumed that both $A/J(A)$ and $B/J(B)$ are separable. But for \cite[Theorem 10.38 (ii),(iii)]{Curtis-Reiner}, it suffices to assume that at least one of $A$ or $B$ is split.)
\end{proof}

\section{Proof of Theorem \ref{theorem:vertices imply sources}}\label{section:Proof of Puig's Corollary 7.4}

\begin{lemma}\label{lemma:direct summand of OG}
	Let $G$ be a finite group and let $P,Q$ be $p$-subgroups of $G$. Then any indecomposable direct summand $M$ of ${\rm Res}_{P\times Q}^{G\times G}(\O G)$ is isomorphic to $\Ind_{\Delta\varphi}^{P\times Q}(\O)\cong\O P\otimes_{\O R} {}_\varphi(\O Q)$ for some subgroup $R$ of $P$ and some injective group homomorphism $\varphi:R\to Q$. In particular, any vertex of $M$ has order $|R|$;  if $M$ has a vertex of order $|P|$, then $M\cong {}_\varphi\O Q$ for some injective  group homomorphism $\varphi:P\to Q$.
\end{lemma}	

\begin{proof}
	Since $\O G\cong \Ind_{\Delta G}^{G\times G} (\O)$ (see e.g. \cite[Corollary 2.4.5]{Lin18a}), by Mackey's formula, $M$ is isomorphic to a direct summand of ${\rm Ind}_{(P\times Q)\cap {}^t(\Delta G)}^{P\times Q}(\O)$ for some $t\in G\times H$. The subgroup $(P\times Q)\cap {}^t(\Delta G)$ is a twisted diagonal subgroup of $P\times Q$, hence of the form $\Delta \varphi$, for some group isomorphism $\varphi:R\cong S$ between subgroups $R\leq P$ and $S\leq Q$. By Green's indecomposability theorem, $\Ind_{\Delta \varphi}^{P\times Q}(\O)$ is indecomposable, hence we have $M\cong \Ind_{\Delta\varphi}^{P\times Q}(\O)\cong \O P\otimes_{\O R} {}_\varphi(\O Q)$, where the second isomorphism sends $(x,y)\otimes 1$ to $x\otimes y^{-1}$ for all $x\in P$ and $y\in Q$.
\end{proof}


If $k$ is algebraically closed, then the next proposition is \cite[Theorem 6.9]{Puig1999}; see \cite[Lemma 3.2]{L25} for an alternating proof using simplified notation. We extend this result to arbitrary $k$ with an even simpler proof. As in \cite[Lemma 3.2]{L25}, we use Bouc's Theorem \cite[Theorem 1.1]{Bouc:bisets}.

\begin{proposition}\label{prop: P and Q are defect groups}
Keep the notation of \ref{notation:stable equ of Morita type}. Let $P:=p_1(X)$ and $Q=p_2(X)$. Then $P$ is a defect group of $b$ and $Q$ is a defect group of $c$.  
\end{proposition}
\begin{proof}
Denote by $c^\circ$ the image of $c$ in $\O H$ under the anti-automorphism of $\O H$ send any $h\in H$ to $h^{-1}$. Since any vertex of $M$ is contained in a defect group of the block $b\otimes c^\circ$ of $\O(G\times H)$, we may assume that $X\leq D\times E$, where $D$ is a defect group $b$ and $E$ is a defect group of $c$. Let the notation $X^\vee$ denote the subgroup $\{(h,g)\mid(g,h)\in X\}$ of $H\times G$. Then $M^*$ has an $\O X^\vee$-source $V^*$. By \cite[Theorem 1.1]{Bouc:bisets}, if $Y\leq G\times G$ is a vertex of any indecomposable direct summand of $M\otimes_{\O Hc}M^*$, then $|p_1(Y)|\leq |p_1(X)|=|P|$ . Since $M\otimes_{\O Hc} M^*$ has a direct summand isomorphic to $\O Gb$, which has $\Delta D$ as a vertex, it follows that $|P|=|D|$, and hence $P=D$. A similar argument applied to $M^*\otimes_{\O Gb}M$ shows that $Q=E$.
\end{proof}

\begin{proposition}\label{prop: source idempotents}
Keep the notation of Proposition \ref{prop: P and Q are defect groups}. The bimodule $M$ is isomorphic to a direct summand of $$\O Gi\otimes_{\O P}{\rm Ind}_X^{P\times Q}(V)\otimes_{\O Q} j\O H$$ for some source idempotents $i\in (\O Gb)^P$ and $j\in (
\O Hc)^Q$.
\end{proposition}

\begin{proof}
Since $M$ has an $\O X$-source $V$, $M$ is isomorphic to a direct summand of 
$$b\Ind_X^{G\times H}(V)c\cong b\Ind_{P\times Q}^{G\times H}\Ind_X^{P\times Q}(V)c\cong \O Gb\otimes_{\O P}{\rm Ind}_X^{P\times Q}(V)\otimes_{\O Q} \O Hc.$$
Since $M$ is indecomposable, there are primitive idempotents $i\in (\O Gb)^P$ and $j\in(\O Hc)^Q$ such that $M$ is isomorphic to a direct summand of 
$$\O Gi\otimes_{\O P}{\rm Ind}_X^{P\times Q}(V)\otimes_{\O Q} j\O H.$$
Now it suffices to show that $\br_P(i)\neq 0$ and $\br_Q(j)\neq 0$. If this is not the case, then the $\O Gb$-$\O P$-bimodule $\O Gi$ has a twisted diagonal vertex of order strictly smaller than $|P|$ or the $\O Q$-$\O Hc$-bimodule $j\O H$ has a twisted diagonal vertex of order strictly smaller than $|Q|$. It follows, by \cite[Theorem 1.1]{Bouc:bisets}, that if $Y\leq G\times H$ is any vertex of any direct summand of $M$, then $|p_1(Y)|<|P|$ or $|p_2(Y)|<|Q|$, a contradiction. 
\end{proof}

\begin{proposition}\label{prop:relation of source algebras}
Keep the notation of Proposition \ref{prop: source idempotents}. Let $A:=i\O Gi$ and $B:=j\O Hj$ be the corresponding source algebras. Assume that $X\cong p_1(X)$. Then the following hold:
\begin{enumerate}[{\rm (i)}]
	\item  There is a primitive idempotent $t\in \End_{\O P\otimes_\O B^{\rm op}}(iMj)$ with $\br_P^{\End_{B^{\rm op}}(iMj)}(t)\neq 0$ such that we have an isomorphism of interior $P$-algebras $$t\End_{B^{\rm op}}(iMj)t\cong e(\End_\O({}_\psi V)\otimes_\O {}_\tau B)e$$ for some group isomorphism $\psi:P\to X$, some surjective group homomorphism $\tau:P\to Q$ and some primitive idempotent $e\in (\End_\O({}_\psi V)\otimes_\O {}_\tau B)^P$ satisfying $\br_P(e)\neq 0$.
	\item As an $\O$-algebra, $t\End_{B^{\rm op}}(iMj)t$ is a symmetric algebra and has an $\O$-basis which is stable under the $P$-conjugation action. 
\end{enumerate}

\end{proposition}
\begin{proof}
(i). Since $M$ induces a stable Morita equivalence of Morita type between $\O Gb$ and $\O Hc$, the $A$-$B$-bimodule $iMj$ induces a stable equivalence of Morita type between $A$ and $B$. By \cite[Corollary 2.12.4]{Lin18a}, we have
\begin{equation}\label{equation:EndiMj iso A plus iU1i}
\End_{B^{\rm op}}(iMj)\cong iMj\otimes_B (iMj)^*\cong A\oplus iU_1i
\end{equation}
 as $A$-$A$-bimodules, where $iU_1i$ is a projective $A\otimes_\O A^{\rm op}$-module (recall that the symbol $U_1$ comes from Notation \ref{notation:stable equ of Morita type}). By Proposition \ref{prop: source idempotents}, the $A$-$B$-bimodule $iMj$ is isomorphic to a direct summand of $A\otimes_{\O P}\Ind_{X}^{P\times Q}(V)\otimes_{\O Q}B.$ Let $e'\in \End_{A\otimes_\O B^{\rm op}} (A\otimes_{\O P}\Ind_{X}^{P\times Q}(V)\otimes_{\O Q}B)$ be a projection of $A\otimes_{\O P}\Ind_{X}^{P\times Q}(V)\otimes_{\O Q}B$ onto $iMj$. Thus
\begin{equation}\label{equation:End(iMj)}
\End_{B^{\rm op}}(iMj)\cong e'(\End_{B^{\rm op}} (A\otimes_{\O P}\Ind_{X}^{P\times Q}(V)\otimes_{\O Q}B))e'
\end{equation}
as $\O$-algebras and as $A$-$A$-bimodules. Through the group homomorphism $P\to A^\times$, we can regard both sides of (\ref{equation:End(iMj)}) as interior $P$-algebras.
Since $A(P)\neq 0$, by (\ref{equation:EndiMj iso A plus iU1i}) we have 	$\End_{B^{\rm op}}(iMj)(P)\neq 0$,
 hence we can choose a primitive idempotent $t\in \End_{\O P\otimes_\O B^{\rm op}}(iMj)$ with $\br_P(t)\neq 0$ and a primitive idempotent $e''\in \End_{\O P\otimes_\O B^{\rm op}} (A\otimes_{\O P}\Ind_{X}^{P\times Q}(V)\otimes_{\O Q}B)$ with $\br_P(e'')\neq 0$, such that 
$$t\End_{B^{\rm op}}(iMj)t\cong e''(\End_{B^{\rm op}} (A\otimes_{\O P}\Ind_{X}^{P\times Q}(V)\otimes_{\O Q}B))e''.$$
 
So there is an indecomposable summand $W$ of $A$ as an $(\O P,\O P)$-bimodule and an $\End_{\O P\otimes_\O B^{\rm op}}(A\otimes_{\O P}\Ind_{X}^{P\times Q}(V)\otimes_{\O Q}B)$-conjugate $e$ of $e''$, such that 
$$ e''(\End_{B^{\rm op}}(A\otimes_{\O P}\Ind_{X}^{P\times Q}(V)\otimes_{\O Q}B))e''\cong e(\End_{B^{\rm op}}(W\otimes_{\O P}\Ind_{X}^{P\times Q}(V)\otimes_{\O Q}B))e$$
as interior $P$-algebras. By Lemma \ref{lemma:direct summand of OG} the order of a vertex of $W$ has order at most $|P|$. The condition $\br_P(e)\neq 0$ forces $W$ to have a vertex of order $|P|$; see \cite[Theorem 2.6.2 (i),(v)]{Lin18a}.

Denote by $\varphi$ the surjective group homomorphism $p_2\circ p_1^{-1}:P\to Q$ (recall that $p_1$ denotes the projection $X\to P$, and $p_2$ denotes the projection $X\to Q$). One checks that we have an $(\O P,B)$-bimodule isomorphism ${\rm Ind}_{X}^{P\times Q}(V)\otimes_{\O Q} B\cong {}_{p_1^{-1}}(V)\otimes_\O {}_\varphi B$ sending $((x,y)\otimes v)\otimes m$ to $p_1^{-1}(x)v\otimes y^{-1}\varphi(x)m$, with inverse sending $v\otimes m$ to $((1,1)\otimes v)\otimes m$, where $x\in P$, $y\in Q$, $v\in V$ and $m\in B$. Hence we have an isomorphism 
$$e(\End_{B^{\rm op}}(W\otimes_{\O P}\Ind_{X}^{P\times Q}(V)\otimes_{\O Q}B))e\cong e(\End_{B^{\rm op}}(W\otimes_{\O P} ({}_{p_1^{-1}}V\otimes_\O {}_\varphi B)))e$$
of interior $P$-algebras.
By Lemma \ref{lemma:direct summand of OG}, $W\cong (\O P)_{\rho}$ for some $\rho\in \Aut(P)$. So we have an isomorphism
$$e(\End_{B^{\rm op}}(W\otimes_{\O P} ({}_{p_1^{-1}}V\otimes_\O {}_{\varphi}B)))e\cong e(\End_{B^{\rm op}}({}_{\rho^{-1}\circ p_1^{-1}}V\otimes_\O {}_{\rho^{-1}\circ\varphi} B))e$$
of interior $P$-algebras. Write $\psi=\rho^{-1}\circ p_1^{-1}$ and $\tau=\rho^{-1}\circ\varphi$.
It is easy to see that we have also an isomorphism 
$$e(\End_{B^{\rm op}}({}_\psi V\otimes_\O {}_\tau B))e\cong e(\End_\O({}_\psi V)\otimes_\O {}_\tau B)e$$
of interior $P$-algebras, proving (i).

(ii) Since matrix algebras are symmetric, the $\O$-algebra $\End_\O({}_\psi V)$ is symmetric. Since a source algebra is symmetric, the $\O$-algebra ${}_\tau B$ is symmetric. Then by \cite[Theorems 2.11.12 (i) and 2.11.11]{Lin18a}, the $\O$-algebra $e(\End_\O({}_\psi V)\otimes_\O {}_\tau B)e$ is symmetric. Hence by statement (i), $t\End_{B^{\rm op}}(iMj)t$ is symmetric. Consider the $A$-$A$-bimodule isomorphism (\ref{equation:EndiMj iso A plus iU1i}). Since $iU_1i$ is a projective $A$-$A$-bimodule, it has a $(P\times P)$-stable $\O$-basis under left and right multiplication. Hence $A\oplus iU_1i$ has a $(P\times P)$-stable $\O$-basis. It follows that $t\End_{B^{\rm op}}(V)t$ has a $(P\times P)$-stable, and hence a $\Delta P$-stable $\O$-basis. 
\end{proof}

\begin{theorem}[{\cite[Theorem 7.2]{Puig1999}}]\label{theorem: Puig 7.2}
	Assume that $k$ is algebraically closed. Let $P$ be a finite $p$-group, and let $A$ and $B$ be finitely generated $\O$-free interior $P$-algebras such that $1_A$ and $1_B$ are primitive in $A^P$ and $B^P$ respectively. Let $V$ be a finitely generated indecomposable $\O$-free $\O P$-module and write $S:=\End_\O(V)$. Assume that there is an interior $P$-algebra embedding $g:A\to S\otimes_\O B$. If $P$ stabilises by conjugation $\O$-bases of $A$ and $B$, $A$ admits a nondegenerate symmetric $\O$-linear form $\mu:A\to \O$ and $A(P)\neq 0$, then $V$ is an endopermutation $\O P$-module. 
\end{theorem}

A more detailed proof of this theorem using simplified terminology is given in Section \ref{section:proof of Puig's Theorem 7.2}.

\begin{proof}[Proof of Theorem \ref{theorem:vertices imply sources}.] 
It remains to prove (ii)$\Rightarrow$(iv). Suppose that (ii) holds. Then we have an isomorphism of interior $P$-algebras 
	$$t\End_{B^{\rm op}}(iMj)t\cong e(\End_\O({}_\psi V)\otimes_\O {}_\tau B)e$$ 
as showed in Proposition \ref{prop:relation of source algebras}. In other words, we have an embedding 
	$$t\End_{B^{\rm op}}(iMj)t\to \End_\O({}_\psi V)\otimes_\O {}_\tau B$$ 
of interior $P$-algebras. By Proposition \ref{prop:relation of source algebras} (ii),  as an $\O$-algebra, $t\End_{B^{\rm op}}(iMj)t$ is a symmetric algebra and has an $\O$-basis which is stable under the $P$-conjugation action. By the choice of $t$, we have $(t\End_{B^{\rm op}}(iMj)t)(P)\neq 0$. Since $B$ is a source $Q$-algebra, ${}_\tau B$ has a $P$-stable $\O$-basis as well. Since the homomorphism $\tau:P\to Q$ is surjective and since $1_B=j$ is primitive in $B^Q$, $1_B$ is primitive in $({}_\tau B)^P$ as well. Now by Theorem \ref{theorem: Puig 7.2}, ${}_\psi V$ is an endopermutation $\O P$-module. Since $\psi$ is a group isomorphism, $V$ is an endopermutation $\O X$-module.
\end{proof}

\section{On vertices and sources, and proof of Theorem \ref{theorem:descent to finite fields}}\label{section:Descending to finite fields}


For the proofs of next two lemmas, we are grateful to Prof. Conghui Li for recommending \cite[Theorem 12.6]{AF} to us.

\begin{lemma}[a generalisation of {\cite[Chapter 3, Lemma 4.14]{Feit}}]\label{lemma:vertex is invariant under scalar extensions}
	Let $\O'$ be a complete discrete valuation ring containing $\O$ which is free as an $\O$-module (possibly having infinite $\O$-rank). Let $G$ be a finite group and $M$ a finitely generated indecomposable $\O G$-module. Let $P$ be a vertex of $M$. Then $P$ is a vertex of every indecomposable direct summand of the $\O'G$-module $\O'\otimes_\O M$. 
\end{lemma}

\begin{proof}
	Let $U'$ be an indecomposable direct summand of the $\O'G$-module $\O'\otimes_\O M$. Since $M$ is isomorphic to a direct summand of $\Ind_P^G{\Res}_P^G(M)$ (see e.g. \cite[Theorem 2.6.2]{Lin18a}), $U'$ is isomorphic to a direct summand of $\Ind_P^G{\Res}_P^G(\O'\otimes_\O M)$. So $U'$ is relatively $P$-projective. If $P$ is not a vertex of $U'$, then there is a proper subgroup $Q$ of $P$ which is a vertex of $U'$. Let $I$ be an $\O$-basis of $\O'$ (which is possibly infinite). Since $U'$ is isomorphic to a direct summand of $\O'\otimes_\O M$, ${\rm Res}_{\O G}^{\O'G}(U')$ is isomorphic to a direct summand of $\Res_{\O G}^{\O'G}(\O'\otimes_\O M)\cong \oplus_{i\in I} M$. Since $\End_{\O G}(M)$ is a local algebra (see e.g. \cite[Corollary 4.4.7]{Lin18a}), by \cite[Theorem 12.6]{AF} $M$ is isomorphic to a direct summand of ${\rm Res}_{\O G}^{\O'G}(U')$. Since $Q$ is a vertex of $U'$, $U'$ is isomorphic to a direct summand of $\Ind_Q^G{\Res}_Q^G(U')$. Hence $M$ is isomorphic to a direct summand of $\Ind_Q^G{\Res}_Q^G({\rm Res}_{\O G}^{\O'G}(U'))$, which implies that $M$ is relatively $Q$-projective, a contradiction.
\end{proof}

\begin{lemma}[a generalisation of {\cite[Lemma 5.2]{Kessar_Linckelmann}}]\label{lemma:descent of sources}
	Let $\O'$ be a complete discrete valuation ring containing $\O$ which is free as an $\O$-module (possibly having infinite $\O$-rank). Let $G$ be a finite group and $M$ a finitely generated indecomposable $\O G$-module. Let $P$ be a vertex of $M$. Let $U'$ be an indecomposable direct summand of the $\O'G$-module $\O'\otimes_\O M$ and $V'$ an $\O'P$-source of $U'$. Suppose that $V'\cong \O'\otimes_\O V$ for some $\O P$-module $V$. Then $V$ is an $\O P$-source of $M$. Moreover, every indecomposable direct sumand of $\O'\otimes_\O M$ has $V'$ as a source.
\end{lemma}

\begin{proof}
	Let $I$ be an $\O$-basis of $\O'$ (which is possibly infinite). Since $U'$ is isomorphic to a direct summand of $\O'\otimes_\O M$, ${\rm Res}_{\O G}^{\O'G}(U')$ is isomorphic to a direct summand of $\Res_{\O G}^{\O'G}(\O'\otimes_\O M)\cong \oplus_{i\in I} M$. Since $\End_{\O G}(M)$ is a local algebra, by \cite[Theorem 12.6]{AF} any indecomposable direct summand of ${\rm Res}_{\O G}^{\O'G}(U')$ is isomorphic to $M$. By assumption, $V'\cong \O'\otimes_\O V$, hence ${\rm Res}_{\O P}^{\O'P}(V')$ is isomorphic to $\Res_{\O P}^{\O'P}(\O'\otimes_\O V)\cong \oplus_{i\in I} V$. Since $\End_{\O P}(V)$ is a local algebra, again by \cite[Theorem 12.6]{AF} any indecomposable direct summand of ${\rm Res}_{\O P}^{\O'P}(U')$ is isomorphic to $V$. Since $V'$ is isomorphic to a direct summand of ${\rm Res}_P^G(U')$, $\Res_{\O P}^{\O'P}(V')$ is isomorphic to a direct summand of $\Res_P^G(\Res_{\O G}^{\O'G}(U'))$. It follows that $V$ is isomorphic to a direct summand of a direct sum of (possibly infinite) copies of $\Res_P^G(M)$. Using \cite[Theorem 12.6]{AF} again, we see that $V$ is isomorphic to a direct summand of $\Res_P^G(M)$. By Lemma \ref{lemma:vertex is invariant under scalar extensions}, $P$ is a vertex of $V$, and therefore, $V$ is a source of $M$, proving the first statement. Since $M$ is isomorphic to a direct summand of $\Ind_P^G(V)$, $\O'\otimes_\O M$ is isomorphic to a direct summand of $\Ind_P^G(V')$. This implies the second statement.
\end{proof}

\begin{lemma}[a slight generalisation of {\cite[Lemma 6.4]{Kessar_Linckelmann}}]\label{lemma:Kessar Linckelmann lemma 6.4}
	Let $G$ be a finite group. Let $k'$ be an extension of $k$. Let $b$ be a block of $kG$ and $b'$ a block $k'G$ such that $bb'\neq0$.  Write $b'=\sum_{g\in G} \alpha_g g$, where $\alpha_g\in k'$ and let $k[b']$ be the smallest subfield of $k'$ containing $k$ and the coefficients $\{\alpha_g\mid g\in G\}$. The following hold:
	\begin{enumerate}[{\rm (i)}]
		\item Then $k[b']$ is a finite Galois extension of $k$.
		\item The block decomposition of $b$ in $k'G$ is $b=\sum_{\sigma\in \Gal(k[b']/k)}\sigma(b')$, where $\sigma(b')=\sum_{g\in G} \sigma(\alpha_g) g$.
		\item A $p$-subgroup $P$ of $G$ is a defect group of $b$ if and only if it is a defect group of $b'$.
	\end{enumerate}
\end{lemma}

\begin{proof}
	(i) Let $\bar{k'}$ be an algebraic closure of $k'$. Let $n=\exp(G)$, the smallest positive integer such that $g^n=1$ for all $g\in G$. Let $\omega$ be a primitive $n_{p'}$-th root of unity in $\bar{k'}$, where $n_{p'}$ is the $p'$-part of $n$. It is well-known that $k[\omega]$ is a finite Galois extension of $k$. Since $k[\omega]$ is a splitting field of $G$, $k[b']$ is a subfield of $k[\omega]$, and hence $k[b']$ is a finite Galois extension of $k$. 
	
	(ii) Let $\sigma\in \Gal(k[b']/k)$. Then $\sigma(b)$ is a block of $k'G$ satisfying $\sigma(b')b=\sigma(b'b)=\sigma(b')$. Hence $\sigma(b')$ appears in a block decomposition of $b$ in $k'G$. By the definition of $\Gal(k[b']/k)$, if $\sigma(b')=b'$, then $\sigma=1$. Hence if $\sigma\neq \tau\in \Gal(k[b']/k)$, then $\sigma(b')\neq \tau(b')$. Now we have 
	$b(\sum_{\sigma\in \Gal(k[b']/k)}\sigma(b'))=\sum_{\sigma\in \Gal(k[b']/k)}\sigma(b').$ Since $k'[b]/k$ is a finite Galois extension, $\sum_{\sigma\in \Gal(k[b']/k)}\sigma(b')$ is a central idempotent in $kG$. Since $b$ is primitive in $kG$, we have $b=\sum_{\sigma\in \Gal(k[b']/k)}\sigma(b')$.
	
	(iii) Write $\Gamma=\Gal(k[b']/k)$. We identify $\br_P^{kG}:kG\to (kG)(P)$ with the map $(kG)^P\to kC_G(P)$ sending any $\sum_{g\in G}\beta_gg\in (kG)^P$ to $\sum_{g\in C_G(P)}\beta_gg$. By (ii) we have $$\br_P^{kG}(b)=\br_P^{k'G}(b)=\br_P^{k'G}(\sum_{\sigma\in \Gamma}\sigma(b'))=\sum_{\sigma\in \Gamma}\br_P^{k'G}(\sigma(b'))=\sum_{\sigma\in \Gamma}\sigma(\br_P^{k'G}(b')).$$
	If $\br_P^{k'G}(b')\neq 0$, then for any $\sigma\in \Gamma$, $\sigma(\br_P^{k'G}(b'))\neq 0$, and in this case they are orthogonal because $b'$ and $\sigma(b')$ are orthogonal. Therefore, 
	$$\br_P^{kG}(b)\neq 0~~\Longleftrightarrow~~\sum_{\sigma\in \Gamma}\sigma(\br_P^{k'G}(b'))\neq 0~~\Longleftrightarrow~~\br_P^{k'G}(b')\neq 0,$$
	whence statement (iii).
\end{proof}

\begin{theorem}\label{prop:scalar extension of stable equ}
	Keep the notation of \ref{notation:stable equ of Morita type}. Assume that $\O=k$. Let $\bar{k}$ be an algebraic closure of $k$. Let $\t{b}$ and $\t{c}$ be blocks of $\bar{k}G$ and $\bar{k}H$ respectively, such that $b\t{b}=\t{b}$ and $c\t{c}=\t{c}$. Write $\tilde{b}=\sum_{g\in G} \alpha_g g$, where $\alpha_g\in \bar{k}$ and let $k[\tilde{b}]$ be the smallest subfield of $\bar{k}$ containing $k$ and the coefficients $\{\alpha_g\mid g\in G\}$. Assume that $b$ has a nontrivial defect group. The following hold:
	\begin{enumerate}[{\rm (i)}]
		\item We have $k[\t{b}]=k[\t{c}]$, where $k[\t{c}]$ is defined similarly.
		\item Let $k'$ be any extension of $k[\t{b}]$. There is a indecomposable direct summand $M'$ of $k'\otimes_k M$ inducing a stable equivalence of Morita type between $k'G\t{b}$ and $k'H\sigma(\t{c})$ for some $\sigma\in \Gal(k[\t{c}]/k)$. Moreover, $M'$ has $X$ as a vertex.
		\item Keep the notation in (ii). Let $V'$ be a $kX$-source of $M'$. If $V'\cong k'\otimes_k Y$ for some $kX$-module $Y$, then $M$ has $Y$ as a source.
	\end{enumerate}
\end{theorem}

\begin{proof}
	For the purpose of statement (i), we may assume that $k$ is finite (because we may replace $k$ by the smallest subfield of $k$ containing $b$ and $c$). Since $M$ induces a stable equivalence of Morita type between $kGb$ and $kHc$, the $\bar{k}Gb$-$\bar{k}Hc$-bimodule $\bar{k}\otimes_k M$ induces a stable equivalence of Morita type between $\bar{k}Gb$ and $\bar{k}Hc$. Since $b$ has a nontrivial defect group, by Proposition \ref{prop: P and Q are defect groups}, $c$ has nontrivial defect group. By Lemma \ref{lemma:Kessar Linckelmann lemma 6.4} (iii), any block of $\bar{k}Gb$ or $\bar{k}Hc$ has nontrivial defect group. Then by a result of Liu \cite[Proposition 2.1]{Liu}, $\bar{k}Gb$ and $\bar{k}Hc$ have the same number of indecomposable direct summands. On the other hand, by Lemma \ref{lemma:Kessar Linckelmann lemma 6.4} (ii), $\bar{k}Gb$ has $|\Gal(k[\t{b}]/k])|$ indecomposable direct summands, and $\bar{k}Hc$ has $|\Gal(k[\t{c}]/k])|$ indecomposable direct summands, which forces $|\Gal(k[\t{b}]/k])|=|\Gal(k[\t{c}]/k])|$. Since $k[\t{b}]/k$ and $k[\t{c}]/k$ are Galois extensions (see Lemma \ref{lemma:Kessar Linckelmann lemma 6.4} (i)), we obtain $|k[\t{b}]:k|=|k[\t{c}]:k|$. Since $k$ is a finite field, this implies $k[\t{b}]=k[\t{c}]$, proving (i).

	Now we drop the assumption in (i) that $k$ is finite. By Lemma \ref{lemma:Kessar Linckelmann lemma 6.4} (ii), $k'Gb=\oplus_{\sigma\in\Gal(k[\tilde{b}]/k)}k'G\sigma(\t{b})$ and $k'Hc=\oplus_{\sigma\in\Gal(k[\tilde{c}]/k)}k'G\sigma(\t{c})$. Now the $k'Gb$-$k'Hc$-bimodule $k'\otimes_k M$ induces a stable equivalence of Morita type between $k'Gb$ and $k'Hc$. By  \cite[Theorem 2.2]{Liu} (see \cite[Proposition 5.4.4]{Zimmermann} for a slightly general version), there is an indecomposable direct summand $M'$ of $k'\otimes_k M$ inducing a stable equivalence of Morita type between $k'G\t{b}$ and $k'H\sigma(\t{c})$ for some $\sigma\in \Gal(k[\tilde{c}]/k)$. By Lemma \ref{lemma:vertex is invariant under scalar extensions}, $M'$ has $X$ as a vertex, proving (ii). Statement (iii) follows from Lemma \ref{lemma:descent of sources}.
\end{proof}

\begin{proof}[Proof of Theorem \ref{theorem:descent to finite fields}]
If $b$ has a trivial defect group, then by Proposition \ref{prop: P and Q are defect groups}, $X$ is trivial and there is nothing to prove. So we may assume that $b$ has a nontrivial defect group.	Under assumption of Theorem \ref{theorem:descent to finite fields}, we can use the notation in Theorem \ref{prop:scalar extension of stable equ}. By Theorem \ref{prop:scalar extension of stable equ}, there is an indecomposable direct summand $M'$ of $\bar{k}\otimes_k M$ inducing a stable equivalence of Morita type between $\bar{k}G\t{b}$ and $\bar{k}H\sigma(\t{c})$ for some $\sigma\in \Gal(k[\tilde{b}]/k)$. Moreover, $M'$ has $X$ as a vertex. Since $\bar{k}$ is algebraically closed and since $X\cong p_1(X)$ (resp. $X\cong p_2(X)$), by Theorem \ref{theorem:vertices imply sources}, $M'$ has an endopermutation $\bar{k}X$-module $V'$ as a source. By our assumption and by the classification of indecomposable endopermutation modules, $V'$ is defined over $k$; see \cite[Theorem 9.5]{Bouc:The Dade group} or \cite[Theorem 13.3]{Th07}. Hence there exists an endopermutation $kX$-module $Y$ such that $V'\cong \bar{k}\otimes_k Y$. By Theorem \ref{prop:scalar extension of stable equ} (iii), $Y$ is a source of $M$. It follows that any source of $M$ is an endopermutation module.
\end{proof}

\section{Proof of Theorem \ref{theorem: Puig 7.2}}\label{section:proof of Puig's Theorem 7.2}

To prove Theorem \ref{theorem: Puig 7.2}, we need the following lemmas.

\begin{lemma}[{\cite[Lemma 7.16]{Puig1999}}]\label{lemma: Puig 7.16}
With the notation of Theorem \ref{theorem: Puig 7.2}, for any subgroup $Q$ of $P$, the $k$-algebras $A(Q)$, $B(Q)$ and $S(Q)$ are nonzero, and $\mu$ induces a nondegenerate symmetric $k$-linear form $\mu(Q):A(Q)\to k$.
\end{lemma}	

\begin{proof}
Let $X$ be a $P$-stable $\O$-basis of $A$. Since $A(P)\neq 0$, by Lemma \ref{lemma:basis}. $X^P\neq \varnothing$. Hence $X^Q\neq \varnothing$, and this implies $A(Q)\neq 0$. Applying the $Q$-Brauer functor to the embedding $g$, we obtain an embedding 
$$g(Q):A(Q)\to (S\otimes_\O B)(Q)\cong S(Q)\otimes_k B(Q),$$
where the isomorphism holds by the assumption that $B$ has a $P$-stable $\O$-basis; see Lemma \ref{lemma:Puig 7.10}. So $S(Q)$ and $B(Q)$ are nonzero.
The last statement follows from Lemma \ref{lemma: linear form Puig}. 
\end{proof}

\begin{lemma}[{\cite[Lemma 7.17]{Puig1999}}]\label{lemma: Puig 7.17}
Keep the notation of Theorem \ref{theorem: Puig 7.2}. For any subgroup $Q$ of $P$ and any subgroup $R$ of $N_P(Q)$ containing $Q$, we have $S(Q)(R)\neq 0$. Moreover, there exist a local point $\gamma$ of $Q$ on $S$ and a local point $\delta$ of $R$ on $S$ such that $Q_\gamma\leq R_\delta$.  
\end{lemma}

\begin{proof}
The embedding $g:A\to S\otimes_\O B $ induces a $k$-algebra embedding
$$A(Q)(R)\to (S\otimes_\O B)(Q)(R).$$
Since $B$ has a $P$-stable $\O$-basis, we can use Lemma \ref{lemma:Puig 7.10} to obtain the following $k$-algebra embedding
$$A(R)\to S(Q)(R)\otimes_k B(R),$$
which forces $S(Q)(R)\neq 0$. Moreover, since the homomorphism $\alpha_S(Q,R):S(R)\to S(Q)(R)$ is unitary, there exists a primitive idempotent $j\in S^R$ such that $\alpha_S(Q,R)(\br_R^S(j))\neq 0$. By the definition of $\alpha_S(Q,R)$, we have $\alpha_S(Q,R)(\br_R^S(j))=\br_R^{S(Q)}(\br_Q^S(j))$; see \ref{equation: Brauer const. transitivity}. Hence $\br_Q^S(j)\neq 0$. Thus there exists a primitive idempotent $i\in S^Q$ such that $ij=i=ji$ and $\br_Q^S(i)\neq 0$. By definition, both $i$ and $j$ are local idempotents. Let $\gamma$ be the local point of $Q$ on $S$ containing $i$ and $\delta$ the local point of $R$ on $S$ containing $j$. Then by definition $Q_\gamma\leq R_\delta$.
\end{proof}

\begin{lemma}[{\cite[Lemma 7.18]{Puig1999}}]\label{lemma: Puig Lemma 7.18}
Keep the notation of Theorem \ref{theorem: Puig 7.2}. There is an embedding $f:\O\to S^{\rm op}\otimes_\O S$ of $P$-interior algebras, where $S^{\rm op}\otimes_\O S$ is an interior $P$-algebra via the structure homomorphism $P\to (S^{\rm op}\otimes_\O S)^\times$ sending $x$ to $x^{-1}\otimes x$ for any $x\in P$.
\end{lemma}

\begin{proof}
	We have an isomorphism
	\begin{equation}\label{equation:Sop otimes S iso to End(S)}
	S^{\rm op}\otimes_\O S\cong \End_\O(S)
	\end{equation}
	of $\O$-algebras sending $s_1\otimes s_2$ to the map $(s\mapsto s_2ss_1)\in \End_\O(S)$, for any $s_1\in S^{\rm op}$, $s_2\in S$ and $s\in S$. 
 Hence by the isomorphism (\ref{equation:Sop otimes S iso to End(S)}), $\End_\O(S)$ is an interior $P$-algebra with the structure homomorphism
$P \to \End_\O(S)^\times$, $x \mapsto (s\mapsto xsx^{-1})$ for any $x\in P$ and $s\in S$. We regard $S$ as an $\O P$-module via this structure map.  So it suffices to show that the trivial $\O P$-module $\O$ is a direct summand of the $\O P$-module $S$.  By our assumption, under the conjugation action, $A$ and $B$ are permutation $\O P$-modules. Hence any indecomposable direct summand of the $\O P$-module $B$ is isomorphic to $\Ind_Q^P(\O)$ for some subgroup $Q$ of $P$, and therefore any indecomposable direct summand of the $\O P$-module $S\otimes_\O B$ is isomorphic to $\Ind_Q^P(W)$ for some indecomposable direct summand $W$ of $\Res_Q^P(S)$.  Since $A(P)\neq 0$, by Lemma \ref{lemma:basis} we see that $\O$ is isomorphic a direct summand of the $\O P$-module $A$.  Since the embedding $g:A\to S\otimes_\O B$ is also an injective homomorphism of $\O P$-modules, $\O$ is also isomorphic to a direct summand of the $\O P$-module $S\otimes_\O B$. Consequently, at least once we have $Q=P$ and $W=\O$.
\end{proof}

\begin{lemma}[{\cite[Lemma 7.19]{Puig1999}}]\label{lemma: Puig lemma 7.19}
Keep the notation of Theorem \ref{theorem: Puig 7.2}. We have a $k$-algebra isomorphism
$$(S^{\rm op}\otimes_\O S\otimes_\O B)(P)\cong S(P)^{\rm op}\otimes_k S(P)\otimes_k B(P).$$
In particular, $P$ has a unique local point on $S^{\rm op}\otimes_\O S\otimes_\O B$, which has multiplicity one.	
\end{lemma}

\begin{proof}
Applying the second commutative diagram in Lemma \ref{lemma: commutative diagrams} to the homomorphisms ${\rm id}_{S^{\rm op}}:S^{\rm op}\to S^{\rm op}$ and $g: A\to S\otimes_\O B$ instead of $f$ and $g$, respectively, we obtain the following commutative diagram:
$$\xymatrix{(S^{\rm op}\otimes_\O A)(P) \ar[rrr]^{({\rm id}_{S^{\rm op}}\otimes g)(P)~~~} &  & & (S^{\rm op}\otimes_\O S\otimes_\O B)(P) \\ 
	S(P)^{\rm op}\otimes_k A(P)\ar[rrr]^{{\rm id}_{S(P)^{\rm op}}\otimes g(P)~~~~}  \ar[u]^{\alpha_{S^{\rm op},A}(P)}  & &  &  S(P)^{\rm op}\otimes_k (S\otimes_\O B)(P)\ar[u]_{\alpha_{S^{\rm op},S\otimes_\O B}(P)}   } $$ 
Since $A$ and $B$ have $P$-stable $\O$-bases, by Lemma \ref{lemma:Puig 7.10}, $\alpha_{S^{\rm op},A}(P)$ is an isomorphism and $S(P)^{\rm op}\otimes_k (S\otimes_\O B)(P)$ is isomorphic to
$S(P)^{\rm op}\otimes_k S(P)\otimes_k B(P)$. Since $S(P)\neq 0$ (see Lemma \ref{lemma: Puig 7.16}) and $V$ is indecomposable, the unity element of $S(P)$ is primitive in $S(P)$ (and hence also in $S(P)^{\rm op}$); see \ref{void:points} (ii).  Since $1_B$ is primitive in $B^P$, again by \ref{void:points} (ii), the unity element of $B(P)$ is primitive in $B(P)$. Hence the unity element of $S(P)^{\rm op}\otimes_k S(P)\otimes_k B(P)$ is primitive; see e.g. \cite[Lemma 2.2]{KL} (here we used the assumption that $k$ is algebraically closed). This forces the embedding ${\rm id}_{S(P)^{\rm op}}\otimes g(P)$  to be an isomorphism. Since the homomorphism $\alpha_{S^{\rm op},S\otimes_\O B}(P)$ is unitary, by the commutative diagram, the embedding $({\rm id}_{S^{\rm op}}\otimes g)(P)$ must be unitary, and hence an isomorphism. This forces $\alpha_{S^{\rm op},S\otimes_\O B}(P)$ to be an isomorphism, proving the first statement. The last statement follows from the first; see Lemma \ref{lemma:local point isomorphism}.
\end{proof}

\begin{lemma}[{\cite[Lemma 7.20]{Puig1999}}]\label{lemma:Puig 7.20}
Keep the notation of Theorem \ref{theorem: Puig 7.2} and Lemma \ref{lemma: Puig Lemma 7.18}. There is an embedding $g':B\to S^{\rm op}\otimes_\O A$ of interior $P$-algebras such that 
$$({\rm id}_{S^{\rm op}}\otimes g)\circ g'=c_\lambda\circ(f\otimes {\rm id}_B )~~~{\rm and}~~~({\rm id}_S\otimes g')\circ g=c_\nu \circ (f\otimes {\rm id}_A)$$
for some $\lambda\in ((S^{\rm op}\otimes_\O S\otimes_\O B)^P)^\times$ and $\nu\in ((S^{\rm op}\otimes_\O S\otimes_\O A)^P)^\times$, where $c_\lambda$ is the inner automorphism of $S^{\rm op}\otimes_\O S\otimes_\O B$ induced by $\lambda$-conjugation and $c_\nu$ has an analogous meaning.
\end{lemma}

\begin{proof}
Since $A$ has a $P$-stable $\O$-basis, by Lemma \ref{lemma:Puig 7.10}, we have $(S^{\rm op}\otimes_\O A)(P)\cong S(P)^{\rm op}\otimes_k A(P)\neq 0$. Hence by \ref{void:points} (ii), there exists a primitive local idempotent $i\in (S^{\rm op}\otimes_\O A)^P$. By Lemma \ref{lemma:embbeding} (iii), the image of $i$ under the embedding
$${\rm id_{\rm S^{\rm op}}}\otimes g: S^{\rm op}\otimes_\O A\to S^{\rm op}\otimes_\O S\otimes_\O B$$
is a primitive local idempotent in $(S^{\rm op}\otimes_\O S\otimes B)^P$. By Lemma \ref{lemma: Puig lemma 7.19}, $P$ has a unique local point (say $\delta$) on $S^{\rm op}\otimes_\O S\otimes B$. Hence $({\rm id_{\rm S^{\rm op}}}\otimes g)(i)\in \delta$. By Lemma \ref{lemma: Puig Lemma 7.18}, we have another embedding
$$f\otimes {\rm id}_B:B\xrightarrow{\cong} \O\otimes_\O B  \to S^{\rm op}\otimes_\O S\otimes_\O B$$
Since (by the assumption) $1_B$ is primitive in $B^P$, by Lemma \ref{lemma:embbeding} (iii), $(f\otimes {\rm id}_B)(1_B)$ is primitive in  $(S^{\rm op}\otimes_\O S\otimes B)^P$. This implies that $(f\otimes {\rm id}_B)(1_B)\in \delta$, and hence there is $\lambda\in ((S^{\rm op}\otimes_\O S\otimes_\O B)^P)^\times$ such that 
$$({\rm id_{\rm S^{\rm op}}}\otimes g)(i)=\lambda (f\otimes {\rm id}_B)(1_B) \lambda^{-1}.$$
So we have 
$${\rm Im}(c_\lambda\circ(f\otimes {\rm id}_B))=\lambda{\rm Im}(f\otimes {\rm id}_B)\lambda^{-1}=({\rm id}_{S^{\rm op}}\otimes g)(i(S^{\rm op}\otimes_\O A)i).$$
By Lemma \ref{lemma:existence of h such that f=gh} there exists an injective homomormorphism $g'$ of interior $P$-algebras making the diagram 
$$\xymatrix{ B  \ar[rr]^{c_\lambda\circ(f\otimes {\rm id}_B)~~~~~~} \ar@{-->}[rrd]_{g'} \ar@{-->}[d]_{g'} &   &  S^{\rm op} \otimes_\O S\otimes_\O B \\
 i(S^{\rm op}\otimes_\O A)i\ar@{^{(}->}[rr]  &   &  S^{\rm op}\otimes_\O A\ar[u]_{{\rm id}_{S^{\rm op}}\otimes g}}$$ 
   commutative. We have ${\rm Im}(g')=i(S^{\rm op}\otimes_\O A)i$, and hence $g'$ is an embedding, proving the first equality.
   
  To prove the second equality, we tensor the equality $({\rm id}_{S^{\rm op}}\otimes g)\circ g'=c_\lambda\circ(f\otimes {\rm id}_B )$ by ${\rm id}_S$ and obtain
  \begin{equation}\label{equation:Puig 7.20.6}
  ({\rm id}_{S\otimes_\O S^{\rm op}}\otimes g)\circ({\rm id}_S\otimes g')=({\rm id}_S\otimes c_\lambda)\circ({\rm id}_S\otimes f\otimes {\rm id}_B ).
  \end{equation}
  We have two embeddings
  $${\rm id}_S\otimes f:S\cong S\otimes_\O \O \xrightarrow{{\rm id}_S\otimes f} S\otimes_\O S^{\rm op}\otimes_\O S$$ 
  and 
  $$f\otimes {\rm id}_S:S\cong \O \otimes_\O S \xrightarrow{f\otimes{\rm id}_S} S\otimes_\O S^{\rm op}\otimes_\O S$$
  of interior $P$-algebras.
  Since the underlying $\O$-algebras $S$ and $S\otimes_{\O} S^{\rm op}\otimes_\O S$ are isomorphic to matrix algebras, by Proposition \ref{prop: matrix algebras unique embedding}, 
  there exists $\nu'\in (S\otimes_\O S^{\rm op}\otimes_\O S)^\times$ such that ${\rm id}_S\otimes f=c_{\nu'}\circ(f\otimes {\rm id}_S)$ as embeddings of $\O$-algebras. By \cite[Proposition 12.1]{Thevenaz}, we may choose such $\nu'$ in $((S\otimes_\O S^{\rm op}\otimes_\O S)^P)^\times$, and hence ${\rm id}_S\otimes f=c_{\nu'}\circ(f\otimes {\rm id}_S)$ as embeddings interior $P$-algebras. By precomposing each side of (\ref{equation:Puig 7.20.6}) with $g$, we obtain
  \begin{align*}
  	\begin{split}
  ({\rm id}_{S\otimes_\O S^{\rm op}}\otimes g)\circ({\rm id}_S\otimes g')\circ g &=({\rm id}_S\otimes c_\lambda)\circ({\rm id}_S\otimes f\otimes {\rm id}_B )\circ g\\
  &=({\rm id}_S\otimes c_\lambda)\circ(c_{\nu'}\otimes {\rm id}_B)\circ(f\otimes{\rm id}_S\otimes {\rm id}_B )\circ g\\
  &=({\rm id}_S\otimes c_\lambda)\circ(c_{\nu'}\otimes {\rm id}_B)\circ(f\otimes g)\\
  &=({\rm id}_S\otimes c_\lambda)\circ(c_{\nu'}\otimes {\rm id}_B)\circ({\rm id}_{S\otimes_\O S^{\rm op}}\otimes g)\circ(f\otimes {\rm id}_A)\\
  &=c_{\nu''}\circ({\rm id}_{S\otimes_\O S^{\rm op}}\otimes g)\circ(f\otimes {\rm id}_A)
  \end{split}
  \end{align*}
  where $\nu''=(1_S\otimes \lambda)\circ(\nu'\otimes 1_B)\in ((S\otimes_\O S^{\rm op}\otimes_\O S\otimes_\O B)^P)^\times$. Now by \cite[Proposition 12.2 (a)]{Thevenaz}, there exists $\nu\in ((S^{\rm op}\otimes_\O S\otimes_\O A)^P)^\times$, such that $({\rm id}_S\otimes g')\circ g=c_\nu \circ(f\otimes {\rm id}_A)$.
\end{proof}

\begin{lemma}[{\cite[Lemma 7.21]{Puig1999}}]\label{lemma: Puig 7.21}
Keep the notation of Theorem \ref{theorem: Puig 7.2}. Let $Q$ be any subgroup of $P$, set $\overline{S(Q)}=S(Q)/J(S(Q))$ and denote by $n(Q):S(Q)\to \overline{S(Q)}$ the canonical map and by $\overline{g(Q)}$ the composed $N_P(Q)$-algebra homomorphism
$$A(Q)\xrightarrow{g(Q)}(S\otimes_\O B)(Q)\xrightarrow{\alpha_{S,B}(Q)^{-1}} S(Q)\otimes_k B(Q)\xrightarrow{n(Q)\otimes {\rm id}_{B(Q)}} \overline{S(Q)}\otimes_k B(Q).$$
Then $\overline{g(Q)}$ is an embedding.
\end{lemma}

\begin{proof}
Since $g(Q)$ is an embedding (see Lemma \ref{lemma:embbeding} (ii)), it suffices to show that $\overline{g(Q)}$ is injective. Consider the embedding $f:\O\to S\otimes S^{\rm op}$ in Lemma \ref{lemma: Puig Lemma 7.18} (we identify $S\otimes_\O S^{\rm op}$ and $S^{\rm op}\otimes_\O S$). Again by Lemma \ref{lemma:embbeding} (ii) we obtain an embedding $f(Q):k\to (S\otimes_\O S)(Q)$. Since $1$ is a primitive idempotent in $k$, $f(Q)(1)$ is contained in a point of $(S\otimes_\O S^{\rm op})(Q)$; see \cite[Proposition 4.12 (a)]{Thevenaz}. Let $L$ be a simple $(S\otimes_\O S^{\rm op})(Q)$-module corresponding to this point; see \ref{void:points} (i). Denote by $\eta:(S\otimes_\O S^{\rm op})(Q)\to \End_k(L)$ the structure homomorphism of $L$. Denote by $\varphi_Q: (S\otimes_\O S^{\rm op})(Q)\to k$ the map sending $a$ to ${\rm tr}(\eta(a))$, the trace of the $k$-linear transformation $\eta(a)$, for any $a\in (S\otimes_\O S^{\rm op})(Q)$. By elementary linear algebra we see that $\varphi_Q(ab)=\varphi_Q(ba)$ for any $a,b\in(S\otimes_\O S^{\rm op})(Q)$ and $\varphi_Q$ vanishes on nilpotent element of $(S\otimes_\O S^{\rm op})(Q)$. By Lemma \ref{lemma: primitive idempotent trace 1}, we have $\varphi_Q(f(Q)(1))=1$ and hence 
\begin{equation}\label{equation: varphi_Q f(Q)=id}
\varphi_Q\circ f(Q)={\rm id}_k.
\end{equation}
In particular, considering the $k$-algebra homomorphism
$$\varphi_Q\circ \alpha_{S,S^{\rm op}}(Q):S(Q)\otimes_k S(Q)^{\rm op} \to (S\otimes_\O S^{\rm op})(Q)\to k,$$
we have $\varphi_Q(\alpha_{S,S^{\rm op}}(Q)(J(S(Q)\otimes_k S(Q)^{\rm op})))=0$, and therefore $\varphi_Q\circ \alpha_{S,S^{\rm op}}(Q)$ factors through a symmetric $k$-form 
$$\bar{\varphi}_Q:\overline{S(Q)}\otimes_kS(Q)^{\rm op}\to k.$$
In other words, we have $\varphi_Q\circ \alpha_{S,S^{\rm op}}(Q)=\bar{\varphi}_Q\circ (n(Q)\otimes {\rm id}_{S(Q)^{\rm op}})$.

Now we claim that
\begin{equation}\label{equation:Puig 7.21.6}
(\bar{\varphi}_Q\otimes \mu(Q))\circ({\rm id}_{\overline{S(Q)}}\otimes(\alpha_{S^{\rm op}, A}(Q)^{-1}\circ g'(Q)))\circ \overline{g(Q)}=\mu(Q)
\end{equation}
which implies that $\mu(Q)$ vanishes on $\ker(\overline{g(Q)})$, forcing $\ker(\overline{g(Q)})=0$ (see Lemma \ref{lemma: Puig 7.16}).
The rest of this proof is to prove the claim.

First, by the commutativity of the tensor product we have
\begin{align*}
\begin{split} 
&({\rm id}_{\overline{S(Q)}}\otimes (\alpha_{S^{\rm op},A}(Q)^{-1}\circ g'(Q)))\circ \overline{g(Q)}\\
&=({\rm id}_{\overline{S(Q)}}\otimes (\alpha_{S^{\rm op},A}(Q)^{-1}\circ g'(Q)))\circ (n(Q)\otimes {\rm id}_{B(Q)})\circ \alpha_{S^{\rm op},A}(Q)^{-1}\circ g(Q)\\
&=(n(Q)\otimes (\alpha_{S^{\rm op},A}(Q)^{-1}\circ g'(Q))) \circ \alpha_{S^{\rm op},A}(Q)^{-1}\circ g(Q)\\
&=(n(Q)\otimes {\rm id}_{S(Q)^{\rm op}}\otimes {\rm id}_{A(Q)})\circ ({\rm id}_{S(Q)}\otimes \alpha_{S^{\rm op},A}(Q)^{-1})\circ({\rm id}_{S(Q)}\otimes g'(Q))\circ \alpha_{S,B}(Q)^{-1}\circ g(Q).
\end{split}
\end{align*}
Consequently, the left side of (\ref{equation:Puig 7.21.6}) is equal to
$$((\bar{\varphi}_Q\circ(n(Q)\otimes {\rm id}_{S(Q)^{\rm op}}))\otimes \mu(Q))\circ  ({\rm id}_{S(Q)}\otimes \alpha_{S^{\rm op},A}(Q)^{-1})\circ({\rm id}_{S(Q)}\otimes g'(Q))\circ \alpha_{S,B}(Q)^{-1}\circ g(Q)$$
$$=((\varphi_Q\circ\alpha_{S,S^{\rm op}}(Q))\otimes \mu(Q)) \circ  ({\rm id}_{S(Q)}\otimes \alpha_{S^{\rm op},A}(Q)^{-1})\circ({\rm id}_{S(Q)}\otimes g'(Q))\circ \alpha_{S,B}(Q)^{-1}\circ g(Q)$$
$$=(\varphi_Q\otimes \mu(Q))\circ(\alpha_{S,S^{\rm op}}(Q)\otimes {\rm id}_{A(Q)})\circ  ({\rm id}_{S(Q)}\otimes \alpha_{S^{\rm op},A}(Q)^{-1})\circ({\rm id}_{S(Q)}\otimes g'(Q))\circ \alpha_{S,B}(Q)^{-1}\circ g(Q).$$
By Lemma \ref{lemma: commutative diagrams} (iv), we have 
$$\alpha_{S\otimes_\O S^{\rm op},A}(Q)\circ(\alpha_{S,S^{\rm op}}(Q)\otimes {\rm id}_{A(Q)})=\alpha_{S,S^{\rm op}\otimes_\O A}(Q)\circ ({\rm id}_{S(Q)}\otimes \alpha_{S^{\rm op},A}(Q)).$$
By Lemma \ref{lemma:Puig 7.10}, $\alpha_{S\otimes_\O S^{\rm op},A}(Q)$ and $\alpha_{S^{\rm op},A}(Q)$ are invertible, hence we have
$$(\alpha_{S,S^{\rm op}}(Q)\otimes {\rm id}_{A(Q)})\circ  ({\rm id}_{S(Q)}\otimes \alpha_{S^{\rm op},A}(Q)^{-1})=\alpha_{S\otimes_\O S^{\rm op},A}(Q)^{-1}\circ \alpha_{S,S^{\rm op}\otimes_\O A}(Q).$$
Now the left side of (\ref{equation:Puig 7.21.6}) becomes
$$(\varphi_Q\otimes \mu(Q))\circ\alpha_{S\otimes_\O S^{\rm op},A}(Q)^{-1}\circ \alpha_{S,S^{\rm op}\otimes_\O A}(Q)\circ({\rm id}_{S(Q)}\otimes g'(Q))\circ \alpha_{S,B}(Q)^{-1}\circ g(Q)$$
By the second commutative diagram in \ref{lemma: commutative diagrams} (i), we have
$$\alpha_{S,S^{\rm op}\otimes_\O A}(Q)\circ ({\rm id}_{S(Q)}\otimes g'(Q))\circ \alpha_{S,B}(Q)^{-1}=({\rm id}_S\otimes g')(Q).$$
Hence the left side of (\ref{equation:Puig 7.21.6}) becomes
$$(\varphi_Q\otimes \mu(Q))\circ \alpha_{S\otimes_\O S^{\rm op},A}(Q)^{-1}\circ  ({\rm id}_S\otimes g')(Q)  \circ g(Q).$$
By Lemma \ref{lemma:Puig 7.20}, there exists  $\nu\in ((S^{\rm op}\otimes_\O S\otimes_\O A)^P)^\times$ such that
$$({\rm id}_S\otimes g')\circ g=c_\nu \circ(f\otimes {\rm id}_A)$$
where $c_\nu$ is the inner automorphism of $S^{\rm op}\otimes_\O S\otimes_\O A$ induced by $\nu$-conjugation. By the second commutative diagram in Lemma \ref{lemma: commutative diagrams} (i), we have 
$$\alpha_{S\otimes_\O S^{\rm op},A}(Q)\circ (f(Q)\otimes {\rm id}_{A(Q)})=(f\otimes {\rm id}_A)(Q)\circ \alpha_{\O,A}(Q)=(f\otimes {\rm id}_A)(Q)$$
because $\alpha_{\O,A}(Q)={\rm id}_{A(Q)}$. Consequently, the left side of (\ref{equation:Puig 7.21.6}) equals to
\begin{align*}
	\begin{split}
	&(\varphi_Q\otimes \mu(Q))\circ \alpha_{S\otimes_\O S^{\rm op},A}(Q)^{-1}\circ  ({\rm id}_S\otimes g')(Q)  \circ g(Q)\\
	&=(\varphi_Q\otimes \mu(Q))\circ \alpha_{S\otimes_\O S^{\rm op},A}(Q)^{-1}\circ (c_\nu \circ (f\otimes {\rm id}_A))(Q)\\
	&=(\varphi_Q\otimes \mu(Q))\circ \alpha_{S\otimes_\O S^{\rm op},A}(Q)^{-1}\circ c_{\br_Q(\nu)} \circ (f\otimes {\rm id}_A)(Q)\\
	&=(\varphi_Q\otimes \mu(Q))\circ c_{\nu_Q}\circ\alpha_{S\otimes_\O S^{\rm op},A}(Q)^{-1} \circ (f\otimes {\rm id}_A)(Q)\\
	&=(\varphi_Q\otimes \mu(Q))\circ c_{\nu_Q}\circ (f(Q)\otimes {\rm id}_{A(Q)})
	\end{split}
\end{align*}
where $\nu_Q=\alpha_{S\otimes_\O S^{\rm op},A}(Q)^{-1}(\br_Q(\nu))$. Since $\varphi_Q\otimes \mu(Q)$ is a symmetric $k$-form over $(S\otimes_\O S^{\rm op})(Q)\otimes_k A(Q)$, we have
$(\varphi_Q\otimes \mu(Q))\circ c_{\nu_Q}=(\varphi_Q\otimes \mu(Q))$. Hence the left side of (\ref{equation:Puig 7.21.6}) becomes
$$(\varphi_Q\otimes \mu(Q))\circ (f(Q)\otimes {\rm id}_{A(Q)})=(\varphi_Q\circ f(Q))\otimes \mu(Q)={\rm id}_k\otimes \mu(Q)=\mu(Q);$$
where the second equality holds by (\ref{equation: varphi_Q f(Q)=id}). This proves the claim.
\end{proof}

\begin{lemma}[{\cite[7.22.4]{Puig1999}}]
Keep the notation of Theorem \ref{theorem: Puig 7.2}. Let $Q$ be a proper subgroup of $P$ and $R$ a subgroup of $N_P(Q)$ properly containing $Q$. Denote by $\overline{\alpha_S(Q,R)}$ the composed homomorphism
$$S(R)\xrightarrow{\alpha_S(Q,R)} S(Q)(R)\xrightarrow{n(Q)(R)}\overline{S(Q)}(R).$$ Let $\overline{g(Q)}$ be defined as in Lemma \ref{lemma: Puig 7.21}.
Then we have
\begin{equation}\label{equation: Puig 7.22.4}
	\overline{g(Q)}(R)\circ \alpha_A(Q,R)=\alpha_{\overline{S(Q)},B(Q)}(R)\circ(\overline{\alpha_S(Q,R)}\otimes \alpha_B(Q,R))\circ\alpha_{S,B}(R)^{-1}\circ g(R).
\end{equation}
\end{lemma}

\begin{proof}
According to the definition of $\overline{g(Q)}$ (see Lemma \ref{lemma: Puig 7.21}), we have 
$$\overline{g(Q)}(R)=(n(Q)\otimes{\rm id}_{B(Q)})(R)\circ \alpha_{S,B}(Q)(R)^{-1}\circ g(Q)(R).$$
By Lemma \ref{lemma: commutative diagrams} (i), we have 
\begin{equation}\label{equation: Puig 7.22.6 (i)}
	g(Q)(R)\circ \alpha_A(Q,R)=\alpha_{S\otimes_\O B}(Q,R)\circ g(R)
\end{equation}
and 
\begin{equation}\label{equation: Puig 7.22.6 (ii)}
	(n(Q)\otimes{\rm id}_{B(Q)})(R)\circ \alpha_{S(Q),B(Q)}(R)=\alpha_{\overline{S(Q)},B(Q)}(R)\circ(n(Q)(R)\otimes {\rm id}_{B(Q)(R)}).
\end{equation}
By Lemma \ref{lemma: commutative diagrams} (ii), we have 
\begin{equation}\label{equation: Puig 7.22.7}
	\alpha_{S,B}(Q)(R)\circ \alpha_{S(Q),B(Q)}(R)\circ(\alpha_S(Q,R)\otimes\alpha_B(Q,R))=\alpha_{S\otimes_\O B}(Q,R)\circ \alpha_{S,B}(R).
\end{equation}
Hence the left side of (\ref{equation: Puig 7.22.4}) is equal to 
\begin{align*}
	\begin{split}
		&(n(Q)\otimes{\rm id}_{B(Q)})(R)\circ \alpha_{S,B}(Q)(R)^{-1}\circ g(Q)(R)\circ \alpha_A(Q,R)\\
		&=(n(Q)\otimes{\rm id}_{B(Q)})(R)\circ \alpha_{S,B}(Q)(R)^{-1}\circ \alpha_{S\otimes_\O B}(Q,R)\circ g(R)\\
		&=(n(Q)\otimes{\rm id}_{B(Q)})(R)\circ  \alpha_{S(Q),B(Q)}(R)\circ(\alpha_S(Q,R)\otimes\alpha_B(Q,R))\circ \alpha_{S,B}(R)^{-1}  \circ g(R)\\
		&=\alpha_{\overline{S(Q)},B(Q)}(R)\circ(n(Q)(R)\otimes {\rm id}_{B(Q)(R)}) \circ(\alpha_S(Q,R)\otimes\alpha_B(Q,R))\circ \alpha_{S,B}(R)^{-1}  \circ g(R)\\
		&=\alpha_{\overline{S(Q)},B(Q)}(R)\circ(\overline{\alpha_S(Q,R)}\otimes \alpha_B(Q,R))\circ\alpha_{S,B}(R)^{-1}\circ g(R),
	\end{split}
\end{align*}
as claimed; here the first equality holds by (\ref{equation: Puig 7.22.6 (i)}), the second by (\ref{equation: Puig 7.22.7}) and Lemma \ref{lemma:Puig 7.10}, and the third by (\ref{equation: Puig 7.22.6 (ii)}).
\end{proof}

\begin{lemma}[{\cite[Lemma 7.22]{Puig1999}}]\label{lemma: Puig 7.22}
Keep the notation of Theorem \ref{theorem: Puig 7.2}. Any subgroup $Q$ of $P$ has a unique local point on $S$.
\end{lemma}	

\begin{proof}
We will continue to use the notation in Lemma \ref{lemma: Puig 7.21}. We argue by induction on $|P:Q|$. Since $V$ is an indecomposable $\O P$-module with vertex $P$, the statement holds for $Q=P$, hence we may assume that $Q<P$. Let $R$ be a subgroup of $N_P(Q)$ strictly containing $Q$, and set $\bar{R}=R/Q$.
By the inductive hypothesis, $R$ has a unique local point $\delta$ on $S$, and hence $\br_R(\delta)$ is the unique point of $S(R)$ (see \ref{void:points} (ii)). Then by the isomorphism (\ref{equation:direct product decomposition}) and Lemma \ref{lemma:local point isomorphism}, we have
$$S(R)/J(S(R))\cong S(R)(\br_R(\delta))\cong S^R(\delta).$$
Again by the inductive hypothesis, $N_P(Q)$ has a unique local point $\epsilon$ on $S$. By Lemma \ref{lemma: Puig 7.17}, we have 
\begin{equation}\label{equation:Puig 7.22.3}
R_\delta\leq N_P(Q)_\epsilon.
\end{equation}

From (\ref{equation: Puig 7.22.4}) and Lemma \ref{lemma:Puig 7.10} we obtain
\begin{equation}\label{equation:Puig 7.22.8}
\alpha_{\overline{S(Q)},B(Q)}(R)^{-1}\circ\overline{g(Q)}(R)\circ \alpha_A(Q,R)=(\overline{\alpha_S(Q,R)}\otimes \alpha_B(Q,R))\circ\alpha_{S,B}(R)^{-1}\circ g(R).
\end{equation}
Since $\overline{g(Q)}(R)$ is an embedding (see Lemmas \ref{lemma: Puig 7.21} and \ref{lemma:embbeding} (ii)), the left side of (\ref{equation:Puig 7.22.8}) is an embedding, so the right side is an embedding as well. Hence if $i$ is a primitive idempotent of $A(R)$, the idempotent 
$$\bar{\iota}=(\overline{\alpha_S(Q,R)}\otimes \alpha_B(Q,R))(\alpha_{S,B}(R)^{-1}(g(R)(i)))$$ 
is primitive in $\overline{S(Q)}(R)\otimes_k B(Q)(R)$; see Lemma \ref{lemma:embbeding} (iii). On the other hand, the idempotent $\alpha_{S,B}(R)^{-1}(g(R)(i))$ is primitive in $S(R)\otimes_k B(R)$ (see Lemma \ref{lemma:embbeding} (iii)), therefore, by Proposition \ref{prop:bijection between points},
$$\alpha_{S,B}(R)^{-1}(g(R)(i))=a(l\otimes j)a^{-1}$$
for suitable primitive idempotents $l\in S(R)$ and $j\in B(R)$, and a suitable $a\in (S(R)\otimes_k B(R))^\times$. Setting $\bar{l}=\overline{\alpha_S(Q,R)}(l)$ and $\bar{j}=\alpha_B(Q,R)(j)$ and $\bar{a}=(\overline{\alpha_S(Q,R)}\otimes \alpha_B(Q,R))(a)$, then we obtain $\bar{\iota}=\bar{a}(\bar{l}\otimes \bar{j})\bar{a}^{-1}$. Hence $\bar{l}\otimes \bar{j}$ is primitive in $\overline{S(Q)}(R)\otimes_k B(Q)(R)$, which forces $\bar{l}$ to be primitive in $\overline{S(Q)}(R)$ (see Proposition \ref{prop:bijection between points}).

In conclusion, the $k$-algebra homomorphism $\overline{\alpha_S(Q,R)}$ maps at least one primitive idempotent $l$ of $S(R)$ to a primitive idempotent $\bar{l}$ of $\overline{S(Q)}(R)$. Since $S(R)$ has the unique point $\br_R(\delta)$, all primitive idempotents of $S(R)$ are conjugate to $l$, hence $\overline{\alpha_S(Q,R)}$ maps any primitive idempotent of $S(R)$ to a primitive idempotent of $\overline{S(Q)}(R)$. Since $\overline{\alpha_S(Q,R)}$ is unitary, $\overline{\alpha_S(Q,R)}$ maps bijectively a pairwise orthogonal primitive idempotent decomposition of the unity element of $S(R)$ to such a decomposition in $\overline{S(Q)}(R)$. This implies that: 
(1). $\overline{S(Q)}(R)$ has a unique point and consequently, $R$ has a unique local point $\bar{\delta}$ on $\overline{S(Q)}$; see \ref{void:points} (ii). (2). The multiplicity $m_{\br_R(\delta)}$ of $\br_R(\delta)$ on $S(R)$ equals to the multiplicity $m_{\br_R(\bar{\delta})}$ of $\br_R(\bar{\delta})$ on $\overline{S(Q)}(R)$. Equivalently, denoting by $m_\delta$ the multiplicity  of $\delta$ on $S^R$ and by $m_{\bar{\delta}}$ the multiplicity of $\bar{\delta}$ on $\overline{S(Q)}^R$, we have
\begin{equation*}
m_\delta=m_{\bar{\delta}};
\end{equation*}
see Lemma \ref{lemma:local point isomorphism}. Consider $S^Q$, $S(Q)$ and $\overline{S(Q)}$ as $\bar{N}_P(Q)$-algebras, where $\bar{N}_P(Q):=N_P(Q)/Q$ . Then $\delta$ is still a local point of $\bar{R}$ on $S^Q$ and $\bar{\delta}$ is still a local point of $\bar{R}$ on $\overline{S(Q)}$. 
Denote by $F$ the composed $\bar{N}_P(Q)$-algebra homomorphism
$S^Q\xrightarrow{\br_Q^S} S(Q)\xrightarrow{n(Q)}\overline{S(Q)}.$ By definition, we have the following commutative diagram:
$$\xymatrix{ \delta\in S^R=(S^Q)^R\ar[rr]^{~~~~~(\br_Q^{S})^R}  \ar@/^2pc/[rrrr]^{F^R} \ar[d]^{\br_R^S} & & S(Q)^R\ar[rr]^{n(Q)^R} \ar[d]^{\br_R^{S(Q)}} & &  \overline{S(Q)}^R\ar[d]^{\br_R^{\overline{S(Q)}}}\ni \bar{\delta}\\ 
	S(R)\ar[rr]^{\alpha_S(Q,R)}  & & S(Q)(R)  \ar[rr]^{n(Q)(R)}  & &  \overline{S(Q)}(R)  
}$$ 
Hence we have 
\begin{equation}\label{equation:relations of delta and bar delta}
\br_R^{\overline{S(Q)}}(F^R(\delta))\subseteq\br_R^{\overline{S(Q)}}(\bar{\delta}).
\end{equation}
By (\ref{equation:direct product decomposition}) and Lemma \ref{lemma:local point isomorphism}, we have
\begin{equation}\label{equation: Puig 7.22.13}
 \overline{S(Q)}\cong\oplus_{\pi'\in \LP_S(Q)}S^Q(\pi')
 \end{equation} 
 as $N_P(Q)$-algebras.
Since $\overline{S(Q)}(R)=\overline{S(Q)}(\bar{R})$ has a unique point, by Proposition \ref{prop: unique P stable point} (i), there is a unique $\bar{R}$-stable point $\pi\in \LP_S(Q)$ such that 
$$(S^Q(\pi))(\bar{R})\cong \overline{S(Q)}(\bar{R}).$$ 
By the uniqueness of $\pi$, we see that $F(\pi)$ is the unique point of $1$ on $\overline{S(Q)}$ such that $1_{F(\pi)}\leq \bar{R}_{\bar{\delta}}$.

We claim that $\pi$ does not depend on the choice of $R$. Denote by $\bar{\epsilon}$ the unique local point of $N_P(Q)$ on $\overline{S(Q)}$. It suffices to show that $\bar{R}_{\bar{\delta}}\leq \bar{N}_P(Q)_{\bar{\epsilon}}$, because in that case $F(\pi)$ is the unique point on $\overline{S(Q)}$ such that  $1_{F(\pi)}\leq \bar{N}_P(Q)_{\bar{\epsilon}}$. We regard $F$ as a homomorphism of $\bar{N}_P(Q)$-algebras. By the discussion in the last paragraph, for any subgroup $Z$ of $\bar{N}_P(Q)$ and any local point $\bar{\xi}$ of $Z$ on $\overline{S(Q)}$, there is a point $\xi$ of $Z$ on $S^Q$ such that 
$$\br_Z^{\overline{S(Q)}}(F^Z(\xi))\subseteq\br_Z^{\overline{S(Q)}}(\bar{\xi})~~~{\rm and}~~~m_\xi=m_{\bar{\xi}},$$
where $m_\xi$ is the multiplicity of $\xi$ on $(S^Q)^Z$ and $m_{\bar{\xi}}$ is the multiplicity of $\bar{\xi}$ on $\overline{S(Q)}^Z$.
By (\ref{equation: f(P)}), $F(Z)(\br_Z^{S^Q}(\xi))=\br_Z^{\overline{S(Q)}}(F^Z(\xi))\subseteq\br_Z^{\overline{S(Q)}}(\bar{\xi})$. By Lemma \ref{lemma:local point isomorphism} and the equality $m_\xi=m_{\bar{\xi}}$, the multiplicity of $\br_Z^{S^Q}(\xi)$ on $S^Q(Z)$ equals to the multiplicity of $\br_Z^{\overline{S(Q)}}(\bar{\xi})$ on $\overline{S(Q)}$. Now by \cite[Proposition 25.3]{Thevenaz}, the $k$-algebra homomorphism $F(Z)$ is a covering homomorphism. Since $Z$ runs over all subgroups of $\bar{N}_P(Q)$, by \cite[Theorem 25.9]{Thevenaz}, $F$ is a covering homomorphism of $\bar{N}_P(Q)$-algebras. Since $R_\delta\leq N_P(Q)_\epsilon$ (see (\ref{equation:Puig 7.22.3})), by \cite[Proposition 25.6 (b)]{Thevenaz}, we have $\bar{R}_{\bar{\delta}}\leq \bar{N}_P(Q)_{\bar{\epsilon}}$, as claimed.

Now we are ready to prove that $\pi$ is the unique local point of $Q$ on $S$. In the isomorphism (\ref{equation: Puig 7.22.13}), let $e\in \overline{S(Q)}$ be the element corresponding to the unity element of $S^Q(\pi)$. Since $\bar{R}$ fixes $\pi$, it fixes $e$. Now by Proposition \ref{prop: unique P stable point} (ii), $\br_{\bar{R}}^{Z(\overline{S(Q)})}(1-e)=0$. Since $R$ runs over all subgroups of $N_P(Q)$ properly containing $Q$, we obtain 
$$1-e\in \bigcap_{1\neq\bar{R}\leq \bar{N}_P(Q)}\ker(\br_{\bar{R}}^{Z(\overline{S(Q)})}).$$
Since $Z(\overline{S(Q)})$ has an $\bar{N}_P(Q)$-stable $k$-basis (see Lemma \ref{lemma: Z(A) has stable basis}), by Lemma \ref{lemma:intersection of kernels of Brauer homomorphisms} (iv), we have
$$1-e\in Z(\overline{S(Q)})_1^{\bar{N}_P(Q)}\subseteq \overline{S(Q)}_1^{\bar{N}_P(Q)}=n(Q)(S(Q)_1^{\bar{N}_P(Q)})=n(Q)(\br_Q^S(S_Q^P));$$
see e.g. \cite[Proposition 5.4.5]{Lin18a} for the last equality. But since $V$ is an indecomposable $\O P$-module and since $S(P)\neq 0$, zero is the unique idempotent in $S_Q^P$. Hence $e=1$ and $\LP_S(Q)=\{\pi\}.$
\end{proof}

\begin{lemma}[{\cite[Lemma 7.23]{Puig1999}}]\label{lemma: Puig Lemma 7.23}
Keep the notation of Theorem \ref{theorem: Puig 7.2}. For any subgroup $Q$ of $P$, the embedding $g$ induces a bijection between the sets $\LP_A(Q)$ and $\LP_{S\otimes_\O B}(Q)$.
\end{lemma}

\begin{proof}
By \cite[Proposition 15.1 (a),(d)]{Thevenaz}, the embeding $g$ induces an injective map $\LP_A(Q)\to \LP_{S\otimes_\O B}(Q)$, hence $|\LP_A(Q)|\leq |\LP_{S\otimes_\O B}(Q)|$. So it suffices to show that $|\LP_A(Q)|=|\LP_{S\otimes_\O B}(Q)|$.  We have 
$$|\LP_{S\otimes_\O B}(Q)|=|\P((S\otimes_\O B)(Q))|=|\P(S(Q)\otimes_k B(Q))|=|\P(S(Q)|\times |\P(B(Q))|$$
$$=|\P(B(Q))|=|\LP_B(Q)|,$$
where the first and the last equalities hold by \ref{void:points} (ii), the second by Lemma \ref{lemma:Puig 7.10}, the third by Proposition \ref{prop:bijection between points}, and the fourth by Lemma \ref{lemma: Puig 7.22}. In conclusion we have $|\LP_A(Q)|\leq |\LP_B(Q)|$. Since we have another embedding $g':B\to S^{\rm op}\otimes_\O A$ (see Lemma \ref{lemma:Puig 7.20}), by an analogous argument as above, we obtain
$$|\LP_B(Q)|\leq |\LP_{S^{\rm op}\otimes_\O A}(Q)|=|\LP_A(Q)|.$$
This forces $|\LP_A(Q)|=|\LP_{S\otimes_\O B}(Q)|$.
\end{proof}	

\begin{proof}[Proof of Theorem \ref{theorem: Puig 7.2}.]
The goal is to prove that $S$ has a $P$-stable $\O$-basis. By Lemma \ref{lemma: Puig 7.16}, we have $B(P)\neq 0$. Since $B$ has a $P$-stable $\O$-basis, this implies that $\O$ is isomorphic to a direct summand of $B$ as $\O P$-module, where $P$ acts by conjugation on $B$. It follows that $S$ is isomorphic to a direct summand of $S\otimes_\O B$ as an $\O P$-module. So it suffices to prove that $S\otimes_\O B$ has a $P$-stable $\O$-basis. By \cite[Theorem 24.1 (a)]{Thevenaz} (which is originally proved by Puig \cite{Puig1979}), there exists an orthogonal idempotent decomposition $1=\sum_{i\in I}i$ of the unity element of $1\in S\otimes_\O B$, satisfying the following two conditions: (i) For any $i\in I$ and $u\in P$, we have ${}^ui\in I$; (ii) For any $i\in I$, denoting by $P_i$ the stabiliser of $i$ in $P$, then $i$ is a primitive local idempotent in $(S\otimes_\O B)^{P_i}$. Consider the $P$-stable $\O$-module decomposition 
$$S\otimes_\O B=\bigoplus_{i,j\in I} i(S\otimes_\O B)j.$$ 
Now it suffices to show that $i(S\otimes_\O B)j$ has a $(P_i\cap P_j)$-stable $\O$-basis. Since $i$ (resp. $j$) is a primitive local idempotent of $P_i$ (resp. $P_j$) on $S\otimes_\O B$, it follows from Lemma \ref{lemma: Puig Lemma 7.23} that $i=ag(i')a^{-1}$ and $j=cg(j')c^{-1}$ for some idempotents $i'\in A^{P_i}$ and $j'\in A^{P_j}$ and some invertible elements $a\in (S\otimes_\O B)^{P_i}$ and $c\in (S\otimes_\O B)^{P_j}$. Consequently, we obtain $\O(P_i\cap P_j)$-module isomorphisms
$$i'Aj'\cong ai(S\otimes_\O B)jc^{-1}\cong i(S\otimes_\O B)j$$
where the first isomorphism is induced by $g$ and the second is induced by multiplication on the left by $a^{-1}$ and on the right by $c$. Since $A$ has a $P$-stable $\O$-basis, this shows that $i(S\otimes_\O B)j$ has a $(P_i\cap P_j)$-stable $\O$-basis. 
\end{proof}

\bigskip\noindent\textbf{Acknowledgements.}\quad The author is extremely grateful to Professors Zhicheng Feng, Conghui Li and Yuanyang Zhou for some very helpful discussions, and to the referees for their careful reading and invaluable suggestions.

\bigskip
{\footnotesize School of Mathematics and Statistics, Central China Normal University, Wuhan 430079, China 
	
	Email address: xinhuang@mails.ccnu.edu.cn}

\end{document}